\pdfoutput=1 

\documentclass[12pt]{article}
\usepackage{amsmath,latexsym,amssymb,amsthm,enumerate,amsmath,amscd}
\usepackage[mathscr]{eucal}
\usepackage[all,cmtip]{xy}
\usepackage{graphicx}
\theoremstyle{plain}
\newtheorem{theorem}{Theorem}[section]
\newtheorem{proposition}[theorem]{Proposition}
\newtheorem{corollary}[theorem]{Corollary}
\newtheorem{lemma}[theorem]{Lemma}
\theoremstyle{definition}
\newtheorem{definition}[theorem]{Definition}

\newtheorem{remark}[theorem]{Remark}
\newtheorem{note}[theorem]{Note}
\newtheorem{example}[theorem]{Example}

\newtheorem{question}[theorem]{Question}

\theoremstyle{plain}

\theoremstyle{definition}

\newtheorem*{uthm}{Theorem}

\newtheorem*{ack}{Acknowledgment}

\numberwithin{equation}{section}
\numberwithin{table}{section} 
\def\End{\mathrm{End}}

\def\Ob{\mathrm{Ob}}
\def\Mat{\mathrm{Mat}}

\def\cha{\mathrm{char}\ }

\providecommand{\bysame}{\makebox[3em]{\hrulefill}\thinspace}

\def\<{\left<}
\def\>{\right>}

\def\ns{\footnotesize \it}

\def\max{\mathrm{max}}
\def\Ob{\mathrm{Ob}}
\renewcommand{\baselinestretch}{1.7}

\title{Bound on the Jordan type of a generic nilpotent matrix commuting with a given matrix}

\author{Anthony Iarrobino\\[.05in]
{\ns Department of Mathematics,  Northeastern University, Boston, MA 02115, USA.
}\\[.2in]
Leila Khatami\\[.05in]
{\ns Mathematics Department,  Union College, 807 Union Street,  Schenectady, NY 12308, USA. 
}\\[.2in]}

\begin{document}
\maketitle
\begin{abstract}
It is well-known that a nilpotent $n \times n$ matrix $B$ is
determined up to conjugacy by a partition of $n$  formed by the sizes
of the Jordan blocks of $B$. We call this partition the Jordan type of $B$. We
obtain partial results on the following problem: for any
partition $P$ of $n$ describe the type $Q(P)$ of a generic nilpotent
matrix commuting with a given nilpotent matrix of type $P$.
A conjectural description for $Q(P)$ was given by P.~Oblak and restated by
L.~Khatami. In this paper we prove  ``half" of this conjecture by
showing that this conjectural type is less than or equal to $Q(P)$ in
the dominance order on partitions.\footnote{
{{\bf keywords}: Jordan type, nilpotent matrix, commutator, partition.}}
\end{abstract}
\begin{ack} The authors trace beginnings of this work to discussions with Roberta Basili during her visit to Northeastern in summer 2008. An earlier version circulated and submitted in October 2011 listed Roberta Basili as a coauthor. She subsequently posted a different approach\footnote{R. Basili  arXiv:1202.3369, revised in June 2012.} to the problem, and has withdrawn as a coauthor of this article, feeling that her contribution here was not at that level. We appreciate those early discussions, but recognize that our current approach is substantially different. \par
 We are grateful for discussions between the first author and Tomaz 
Ko\v{s}ir and Polona Oblak in June 2008; there Polona Oblak communicated her beautiful conjecture concerning $Q(P)$, which has been a lodestone for our work. We are grateful to Jerzy Weyman, Don King and Bart Van Steirteghem for many discussions about the $P\to Q(P)$ problem and related questions, and to Andrei Zelevinsky for suggestions concerning the exposition. We are very grateful to Bart Van Steirteghem also for his extensive and helpful comments on drafts, that have led to substantial changes. We appreciate the careful reading and thoughtful suggestions of the referees.\par
 \par\medskip

\end{ack}
\section{Introduction}\label{intsec}
It is well known that the nilpotent commutator $\mathcal N_B$ of a Jordan block matrix $B$ whose eigenvalues are in a base field $\sf k$, is a direct sum of the nilpotent commutators corresponding to the generalized eigenspaces of $B$ \cite[p.338]{Ger}. The particular eigenvalue in each plays no further role, so henceforth we assume that $B$ is nilpotent.
\par
We fix an $n$-dimensional vector space $V$ over an infinite field $\mathsf k$, and an $n\times n$ nilpotent Jordan block matrix $B=J_P\in \Mat_n(\mathsf k )$ having $t$ Jordan blocks of sizes $p_i$ given by the partition $P\vdash n, P=(p_1,\ldots ,p_t),\,\  p_1\ge p_2\ge \cdots \ge p_t$. 
 Consider the centralizer $\mathcal C_B\subset \Mat_n(\mathsf{k})\cong\End_{\sf k}(V)$, which is the set of $n\times n$ matrices with entries in $\mathsf{k}$ that commute with $B$, 
and the subvariety $\mathcal N_B$ comprised of those matrices in $\mathcal C_B$ that are nilpotent.  Each element $A$  of $\mathcal N_B$ is in the conjugacy class of a Jordan block matrix $J_{P_A}$ of partition $P_A\vdash n$. We term the partition $P_A$ the \emph{Jordan type} of $A$. It is well known that $\mathcal N_B$ is
an irreducible algebraic variety \cite[Lemma 2.3]{Bas},\cite[Lemma 1.5]{BI}. Thus, there is a unique Jordan type $Q(P)=P_A$ associated to a generic matrix $A\in \mathcal N_B$ -- for $A$ in a suitable Zariski dense open subset of $\mathcal N_B$. And $Q(P)$ is greater in the dominance order \eqref{orbitclosureeq}
than any other Jordan type occurring for elements of $\mathcal N_B$.
Of course, a generic $A\in \mathcal N_B$ is usually not itself a Jordan block matrix.
\begin{question}\label{1.1quest} What is $Q(P)$? Determine $Q(P)$ algorithmically from $P$.
\end{question}  When $P$ is \emph{almost rectangular} -- the maximum part of $P$ minus the smallest part is at most one -- then it is easy to see that  $Q(P)=(n)$, a single block. R.~Basili showed that $Q(P)$ has $r_P$ parts, where $r_P$ is the minimum number of almost rectangular subpartitions $P_1,\ldots ,P_r$  needed for a decomposition $P=P_1\cup \ldots \cup P_r$, where by $P_1\cup P_2$ we mean the partition whose parts are the concatenation of those of $P_1$ and $P_2$ (\cite[Proposition 2.4]{Bas}, see also \cite[Theorem 2.17]{BIK1}).\par
Attached to the partition $P$ is a maximal subalgebra $\mathcal U_B\subset \mathcal N_B$ (Section \ref{UBsec}). A key combinatorial object attached to the partition $P$ and defined from $\mathcal U_B$ is the poset $\mathcal D_P$, which has $n$ elements corresponding to a certain basis ${\sf B}=\{b_1,\ldots ,b_n\}$ of $V$. We regard these as being arranged in $t$ rows: each row corresponds to a part $p_i$ of $P$: the $i$-th row is comprised of the basis of a $B$-invariant subspace of $V$ isomorphic to ${\sf k}[B]/B^{p_i}, i\in\{1,\ldots ,t\}$   (see Definition~\ref{DPdef} below). This poset was defined by P.~Oblak as the digraph associated to the maximal subalgebra $\mathcal U_B$ of $\mathcal N_B$: for $ b,b'\in \sf B$ set $ b\le b'$ in $\mathcal D_P$ if $A_{b,b'}\not=0 $ for $ A$ generic in $\mathcal U_B$, when $A$ is expressed in the basis $\sf B$ \cite{Obl1,BIK1}.\footnote{The poset $\mathcal D_P$ is used implicitly by P. Oblak -- the possible edges in her $(\mathcal N_B,A)$ graphs determine the comparable elements of $\mathcal D_P$. The algebra $\mathcal U_B$ was used in Section 4 of \cite{Obl1} and formally defined and studied in \cite{BIK1}. Our graphs are drawn as the transpose of those in \cite{Obl1}, and are rotated ninety degrees. } Furthermore, any matrix $A\in \mathcal N_B$ is conjugate by a matrix in the centralizer $\mathcal C_B$ to one in $\mathcal U_B$, so we may restrict to $\mathcal U_B$ in determining $Q(P)$ \cite{Bas,TuAi,BIK1}.\par
 P. Oblak \cite{Obl1} for $\cha \sf k=0$ and subsequently the first author and R. Basili for $\sf k$ algebraically closed (unpublished) determined the \emph{index}-- largest part --
$ {\mathrm i}(Q(P))$ of $Q(P)$ in terms of the poset $\mathcal D_P$. Let $P=(\ldots , i^{n_i},\ldots )$ where $i$ has multiplicity $n_i$. A $U$-chain of $\mathcal D_P$ is a maximal chain whose vertices are comprised of those in the rows of $\mathcal D_P$ corresponding to the parts of an almost rectangular (AR) subpartition $P'\subset P$, union two \emph{hooks} -- one from the source and the other to the sink of $\mathcal D_P$ (see Definition \ref{Uchaindef}).  We associate to an AR subpartition $P'=(a^{n_a},(a-1)^{n_{a-1}})$ of $ P$ the invariant $ob (P')$ which is the length -- number of vertices -- of the unique $U$-chain $U_a$ containing $P'$:
\begin{equation}\label{obeq}
ob(P')=\mid U_a\mid = an_a+(a-1)n_{a-1}+\sum_{c>a} 2n_c.
\end{equation}
\begin{theorem}\label{obthm} \cite{Obl1} Let $\cha \sf k=0$.\footnote{See \cite[Theorem 3.3]{BIK1} for the case $\sf k$ algebraically closed.  Corollary \ref{charcor} shows that Theorem~\ref{obthm} implies its analogue for $\sf k$ infinite.} The index ${\mathrm i}(Q(P))$ is the length of the longest chain in $\mathcal D_P$, and is also the length of the longest $U$-chain in $\mathcal D_P$. 
\end{theorem}

\subsection{A New Lower Bound for $Q(P)$.}

 C. Greene, E. R. Gansner, and S.~Poljak associate to any finite poset $\mathcal D$ a partition $\lambda (\mathcal D)$ defined from its chains, as follows \cite{Gre,Gans,Pol,BrFo}. Let $a_{\mathcal D}$ be the minimum number of chains needed to cover $\mathcal D$. Set
$c_0(\mathcal D)=0$, and for every $i\in \{1,\ldots ,a_{\mathcal D}\}$ set
\begin{align} 
&c_i=c_i({\mathcal D} )= \max \{ \# \text{ vertices of $ \mathcal D$  covered by $i$ chains}\},\\ 
& \lambda_i (\mathcal D)=c_i-c_{i-1}.\label{lambdaDeq}
\end{align}

One can construct similarly to $\lambda (\mathcal D_P)$ a possibly different partition $\lambda_U(\mathcal D_P)$ using $s$-$U$-chains in place of arbitrary chains (Definitions \ref{suchaindef}, \ref{lambdaUdef}). P. Oblak had conjectured that $Q(P)$ could be obtained by a recursive process, first picking a maximum-length chain $C_1$ in $\mathcal D_P$, then a maximum length chain $C_2$ in a new, smaller poset $\mathcal D_{P'}$ where the partition $P'={P-C_1}$ is defined through removing $C_1$ from $\mathcal D_P$ and counting the vertices left in  each row (warning: $\mathcal D_{P'}$ does not have the induced partial order from $\mathcal D_P$). And so on  for $r_P$ steps. Then $Q(P)$ is conjecturally the set of lengths of the chains \cite{BKO}.
The second author has shown that any such Oblak process $\mathcal O$ yields a partition $\Ob_{\mathcal O}(P)=(|C_1|,|C_2|,\dots |C_r|)$ satisfying $\Ob_{\mathcal O}(P)=\lambda_U(\mathcal D_P)$ \cite[\S 2]{Kha}. Thus, the Oblak conjecture  for $Q(P)$ is equivalent to a positive answer to
\begin{question}\label{1.4quest} Is $Q(P)=\lambda_U(\mathcal D_P)$?
\end{question}
Recall the {\emph{dominance} or \emph{orbit closure} order on the set of partitions of $n$ \cite{Ger}. Let $P=(p_1,\ldots ,p_t)$ with $p_1\ge \cdots \ge p_t$ and $  P'=(p'_1,\ldots ,p'_{t'})$ with $ p'_1\ge \cdots \ge p'_{t'}$ be partitions of $n$. Then
\begin{equation}\label{orbitclosureeq}
 P\ge P' \Leftrightarrow \forall i, \sum_{k=1}^ip_k\ge  \sum_{k=1}^ip'_k.
\end{equation}
Our main result is \par\noindent
{\bf Theorem \ref{mainthm}}. Let $\sf k$ be an infinite field, then
\begin{equation}\label{maineq1}
Q(P)\ge \lambda_U(\mathcal D_P).
\end{equation} 

To prove this, we first work over a polynomial ring $\sf R$ over $\sf k$ and define in \eqref{Aeq}} a nilpotent matrix $A_{\sf R} \in \Mat_{\sf R}(n)\cong \End_{\sf R}(V\otimes {\sf R})$ which commutes with $B$. We then show  that $P_{A_{\sf R}}\ge \lambda_U(\mathcal D_P)$, when we consider $A_{\sf R}$ as an element of $\Mat_{\sf F}(n)$, with $\sf F$ the quotient field of $\sf R$ (Corollary \ref{maincor}).\par
To prove that $P_{A_{\sf R}}\ge \lambda_U(\mathcal D_P)$, we show in Theorem \ref{sthm}, that for every $s\in \{1,\ldots , r_P\}$ there exist $v_1,\ldots ,v_s\in V$ such that $\dim_{\sf F}\langle {\sf F}[A_{\sf R}]\circ\{v_1,\ldots ,v_s\}\rangle $ is at least the sum $c^s_U(\mathcal D_P)$ of the first $s$ parts of $\lambda_U(\mathcal D_P)$.  Indeed, together with a well-known property of nilpotent matrices (Lemma \ref{ranklem}), this establishes the desired inequality.\par
In turn, the proof of Theorem \ref{sthm} boils down to showing that for a maximal $s$-$U$-chain $\mathfrak A$, a certain $\sf F$ linear map $\pi_{\mathfrak A}=\pi_{\mathfrak A}(U,A_{\sf R})$ defined between a subspace of ${\sf F}[x]\otimes_{\sf k} V\cong {\sf F}[x]\otimes_{\sf F}V_{\sf F}$, and a subspace of $V\otimes {\sf F}$, both of dimension $c_U^s(\mathcal D_P)$, is an isomorphism. The domain of
$\pi_{\mathfrak A}$ is a subspace $\mathcal T_{\mathfrak A}$ of $ {\sf F }[x]\otimes_{\sf F} \langle v_1,\ldots ,v_s\rangle$ where the $v_i$ are the initial vertices of the $s$ component chains of $U_{\mathfrak A}$, and its co-domain is the span of all the vertices covered by $U_{\mathfrak A}$: it has matrix $M_{\mathfrak A}$ in a suitable basis for each (Definition \ref{pidef}).
We show, and this is the heart of the matter, that $\det (M_{\mathfrak A})\not=0$ by an analysis of the sets of chains from the initial vertices $v_i$ to all the vertices covered by $U_{\mathfrak A}$.
A final step in the proof of \eqref{maineq1} is to specialize to $\sf k$ (Theorem \ref{mainthm}).\par
We next state some further results and questions concerning $Q(P)$. In section \ref{DPsec} we define the poset $\mathcal D_P$, the multi-$U$-chains, the homomorphism $\pi_{\mathfrak A}$ and show some properties we will need.  In section \ref{3sec} we show Theorem \ref{mainthm}. We first give a simple example where $P$ is not AR to illustrate naively the problem of determining $Q(P)$.
\begin{example}\label{421ex} Let $P=(4,2,1)$. Since $P=(4)\cup (2,1)$ is a minimal decomposition into almost rectangular subpartitions,  we have $r_P=2$, and we shall see that $Q(P)=(5,2)$. Here the basis
 $\sf B=\{\sf{a,b,c,d,e,f,g}\} $ with $B{\sf a}={\sf b},B{\sf b=c},B{\sf c=d},B{\sf d}=0,B{\sf e=f},B{\sf f=0}, $ and $B\sf g=0$. Since $A $ and $B$ commute, $A\in\mathcal U_B$ is determined by its action on the $B$-cyclic vectors $\{{\sf a,e,g} \}$ of $V$. To obtain a general enough $A$ so that $P_A=(5,2)=Q(P)$  we may take (Figure \ref{421afig})
\begin{equation}\label{421matrix1eq}
A\cdot {\sf a=b+e}, A\cdot {\sf e=c+g}, A\cdot {\sf g=f}.
\end{equation}

\vskip 0.3cm
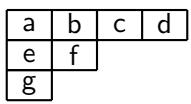
\begin{figure}[hbtp]
\begin{center}
\leavevmode 
\begin{picture}(25,15)(25,-15)
\setlength{\unitlength}{1mm}
\thicklines
\multiput(0,0)(0,-4){2}{\line(1,0){24}}
\multiput(0,-8)(0,-4){2}{\line(1,0){6}}
\multiput(0,-4)(0,-4){2}{\line(1,0){12}}
\multiput(0,0)(6,0){2}{\line(0,-1){12}}
\multiput(12,0)(6,0){1}{\line(0,-1){8}}
\multiput(12,0)(6,0){3}{\line(0,-1){4}}
\put(3,-2){\makebox(0,0){{\small \mbox{$\sf a$}}}}
\put(9,-1.7){\makebox(0,0){{\small \mbox{$\sf b$}}}}
\put(15,-2.1){\makebox(0,0){{\small \mbox{$\sf c$}}}}
\put(3,-10){\makebox(0,0){{\small \mbox{$\sf g$}}}}
\put(3,-6){\makebox(0,0){{\small \mbox{$\sf e$}}}}
\put(9,-6){\makebox(0,0){{\small \mbox{$\sf f$}}}}
\put(21,-1.7){\makebox(0,0){{\small \mbox{$\sf d$}}}}
\end{picture}
\end{center}
\vspace{-1cm}
\protect\caption{$P=(4,2,1), Q(P)=(5,2)\qquad\qquad$.}\label{421afig}
\end{figure} 
\noindent
In  Example \ref{421advex} we apply the proof method of this paper to $P=(4,2,1)$: the endomorphism $A$ above is obtained by substituting $1$ for each of the variables of $\sf R$ in the matrix $A_{\sf R}$ of \eqref{421adveq}. 
\end{example}
In Examples \ref{4221aex} and \ref{2.5ex} below we determine $\lambda_U(P)$ for $P=(4,2,2,1)$ and $P=(5,4,3,3,2,1)$. By Corollary \ref{main2cor}, $Q(P)=\lambda_U(P)$ for these $P$, since $r_P\le 3$. 
\subsection{Some open questions.}

 Recall that the incidence algebra $\mathfrak I (\mathcal D)$ of the $n-$element poset $\mathcal D$ is the algebra of $n\times n$ matrices $M$ satisfying $M_{uv}\in \sf k$ if $u\le v$, and $M_{uv}=0$ if $u\not<v$.

  The nilpotent matrices $\mathcal N(\mathcal D)$ in $\mathfrak I(\mathcal D)$ are those such that $\forall u \,\,m_{uu}=0$.  Suppose that $\mathcal D$ is acyclic, as  is true for the posets $\mathcal D_P$ we  consider. Then these nilpotent matrices have entries $m_{uv}\in \sf k$ that are arbitrary for intervals $[u,v]$ with $u<v$. Then, evidently, $\mathcal N(\mathcal D)$ is an irreducible variety. We have
\begin{uthm}(\cite{Gans,Sa1}, see also \cite[Theorem 6.1]{BrFo}). 
 A. Let $\mathcal D$ be a finite poset and suppose $\cha k=0$. A generic nilpotent matrix $M\in \mathcal N(\mathcal D)$ has Jordan type $P_M =\lambda (\mathcal D)$.\par
B. (\cite{Sa2}, see also \cite[Proof of Theorem 6.1]{BrFo}.\footnote{This is stated in slightly different language for $\sf k=\mathbb C$ in \cite[Theorem 5.16ii]{Sa2}, however the proof there of \eqref{qlambdaeq} does not depend on characteristic, nor require $A$ to have generic entries nor be ``free'' in the language of \cite{Sa2}. Likewise, the proof of Theorem 6.1 in \cite{BrFo}, although stated for ${\sf k}=\mathbb R$, shows \eqref{qlambdaeq} for $\sf k$ infinite.} Let $\sf k$ be an infinite field, let $\mathcal D$ be an acyclic poset, and $M\in \mathcal N(\mathcal D)$. Then
\begin{equation}\label{qlambdaeq}
\lambda (\mathcal D)\ge P_M.
\end{equation} 
\end{uthm}     
The commutator subset $\mathcal C_B\cap \mathfrak I(\mathcal D_P) \subset\mathfrak I (\mathcal D_P)$ of the incidence algebra of $\mathcal D_P$ consists of those $A\in \mathfrak I (\mathcal D_P)$ whose entries $A_{uv}$ satisfy, for $u,v\in \sf B$
\begin{equation}\label{commutatoreq}
u\le v \text { and } Bu, Bv \not=0 \Rightarrow A_{B\cdot u, B\cdot v}=A_{uv}.
\end{equation}
This is a Toeplitz condition on the blocks of $A$ (see \cite[Lemma 2.2]{Bas}, \cite{TuAi}).
\par Since $\mathcal D_P$ is acyclic, the nilpotent matrices $\mathcal N(\mathcal D_P)\subset \mathfrak I(\mathcal D_P)$ form an irreducible family satisfying $\mathcal N(\mathcal D_P)\supset \mathcal U_B=\mathcal C_B\cap \mathcal N(\mathcal D_P)$. By \eqref{qlambdaeq} we have $\lambda (\mathcal D_P)\ge Q(P)$.
\begin{question}\label{1.3quest} Is $Q(P)=\lambda (\mathcal D_P)$?
\end{question}
  We can also ask the seemingly purely combinatorial question
\begin{question}\label{1.5quest}
 Is $\lambda(\mathcal D_P)= \lambda_U(\mathcal D_P)$?
\end{question}
 In view of Theorem \ref{mainthm} and \eqref{qlambdaeq} a positive answer to Question \ref{1.5quest} would also imply  ``yes'' to Questions \ref{1.4quest} and \ref{1.3quest}. In this direction P. Oblak in Theorem \ref{obthm} showed that the index -- largest part --  of $\lambda(\mathcal D_P)$ and of $\lambda_U(\mathcal D_P)$ are the same, and the second author has shown that
the minimum parts of $\lambda(\mathcal D_P)$ and $\lambda_U(\mathcal D_P)$ are the same \cite{Kha2}. Together with Theorem \ref{mainthm} their results imply "Yes" to Question \ref{1.5quest} and Question \ref{1.4quest} when $r_P\le 3$ (Corollary~\ref{3.9cor}).
\par

Even if the Questions above were answered, it could still be a nontrivial combinatorial problem to identify compactly which partitions $P$ satisfy $Q(P)=Q$ for a given partition $Q$. This is discussed in \cite{Obl2}.

\subsubsection{What else is known about $Q(P)$?}
T. Ko\v{s}ir and P.~Oblak showed that the Artinian algebra $\mathcal A=\mathsf{k}[A,B]$ is Gorenstein for general enough $A\in \mathcal U_B$ \cite[Corollary 5]{KO}. The present authors with R. Basili gave sufficient conditions on $A\in \mathcal U_B$ for the algebra $\mathcal A$ to be Gorenstein \cite[Theorem~2.20]{BIK1}. When $\cha \mathsf{k}=0$ or $\cha \mathsf{k}>n$ the partition $Q(P)$ is dual to the Hilbert function $H(\mathsf{k}[A,B])$, viewed as a partition of $n$, for generic $A\in \mathcal N_B$ \cite[Theorem 2.23]{BI}.  From these last results it follows that when $\cha \mathsf{k}=0$ or $\cha\mathsf{k}>n$, $Q(P)$ has parts differing pairwise by at least two, and that $Q(Q(P))=Q(P)$ (\cite[Theorem 6]{KO}, see also \cite[Section 2.5]{BIK1}).\par
 A. Premet, 
G.~McNinch and D.I. Panyushev studied  pairs of commuting nilpotent matrices in the broader context of Lie algebras \cite{Prem,McN,Pan}. V. Baranovsky, R. Basili and others have related the study of commuting nilpotent matrices to the punctual Hilbert scheme of a plane \cite{Bar,Bas,Prem,BI}. R. Guralnick and A.~Sethuranam, K. \v{S}ivic and others have studied commuting pairs and triples of matrices: see \cite{GurSe,SeSi,Si1,Si2} and the references given there. \par

\section{The algebra $\mathcal U_B$ and the poset $ \mathcal D_P$.}\label{DPsec}
\subsection{The algebra $\mathcal U_B$.}\label{UBsec}
We now define a maximal subalgebra $\mathcal U_B$ of $\mathcal N_B$.
Fix an integer $n$ and let $P\vdash n$ be the partition $P=(p_1,p_2,\ldots p_t)$, $
p_1\ge p_2\ge \cdots \ge p_t$ or, in second notation $P=({p_1}^{n_{p_1}},\ldots , i^{n_i}, \ldots ,1^{n_1})$ --
where $n_i $  -- possibly zero -- is the multiplicity of the part $i$ of $P$. Here $n=\sum p_i=\sum i\cdot n_i$. Denote by $S_P=\{ i\mid n_i\not= 0\}$  the set of integers occuring as parts of $P$. For each subpartition $P'$ of $P$ we denote by $\iota (P')\subset S_P$ the set of integers occuring in $P'$.
We have $V=\oplus_{i\in S_P} V_i$, where $V_i$ has a decomposition
\begin{equation}\label{Vsumeq}
V_i=\oplus V_{i,k} \mid 1\le k\le n_i ,
\end{equation}
into cyclic $B$-modules $V_{i,k}$, each of length $i$. The subspace $V_{i,k}$ has a   cyclic vector $(1,i,k)$ and basis 
\begin{equation}\label{Vbasiseq}
 \{(u,i,k)=B^{u-1}(1,i,k), \, 1\le u\le i\}.
\end{equation}
So $V_{i,k}\cong {\sf k}[x]/x^i$ as a ${\sf k}[x]$-module through the action of $B$.
\begin{definition}\label{basisdef} We denote by $\sf B$ the basis of $V$ that is the union of the bases for $V_{i,k}$ defined in \eqref{Vsumeq}, \eqref{Vbasiseq}. 
For each $i\in S_P$ denote by $W_i\subset V_i$ the $n_i$-dimensional subspace of $V$ spanned by the level-$i$ cyclic vectors,
\begin{equation}\label{Wieqn}
W_i=\langle\{ (1,i,k), 1\le k\le n_i \}\rangle ,
\end{equation}
with basis ordered by ``$k$''. Let $W=\oplus_{i\in S_P} W_i$.
\end{definition}
 We have $W\cong V/{\mathrm Im}(B)$, where  $\mathrm {Im}(B)$ denotes the image $B( V)$. Denote by $\kappa_i$ the natural projection:  $V\to W\to W_i$.
Let $\mathcal M_B=\prod_{i\in S_P} \End_{\mathsf k}(W_i)$ and define
\begin{equation}
\varphi_i:  \mathcal C_B\to \End_{\mathsf k} W_i\cong \Mat_{n_i}({\mathsf k}):\, \varphi_i (A)= \kappa_i (A_{|W_i}).
\end{equation}
and
\begin{equation}\label{canprojeq}
\varphi=\prod \varphi_i :\, \mathcal C_B\to \mathcal M_B.
\end{equation}
It is well known that $\varphi$ is, up to an automorphism of $\mathcal M_B$, the canonical projection from $\mathcal C_B$ to its semisimple quotient, with kernel the Jacobson radical $\mathfrak J_B\subset \mathcal C_B$ (see \cite[Lemma 2.3]{Bas},\cite[Theorem 2.3]{BIK1},\cite[Theorem 6]{HW}). Denote by $\mathcal N$ the product $\mathcal N=\prod_{i\in S_P} N(W_i)$ where $N(W_i)\subset \End_{\mathsf k}(W_i)$ are the nilpotent elements;  and by $\mathcal U=\prod_{i\in S_P} U_T(W_i)\subset \mathcal M_B$ the products of the subalgebras of strictly upper triangular elements $U_T(W_i)\subset\End_{\mathsf k}(W_i)$, in the ordered basis \eqref{Wieqn}. Since the Jacobson radical $\mathfrak J_B$ is already comprised of nilpotent elements of $\mathcal C_B$, it follows from \eqref{canprojeq} that the nilpotent commutator $\mathcal N_B$ satisfies 
\begin{equation}
\mathcal N_B=\varphi^{-1}\left( \mathcal N\right).
\end{equation}
We define
\begin{equation}\label{UBeq}
\mathcal U_B=\varphi^{-1} (\mathcal U).
\end{equation}
For $v\in V$ we denote by $< v,(u,i,k)>$ the coefficient of $v$ on $(u,i,k)$, when $v$ is written in the basis $\sf B$ of Definition \ref{basisdef}.
\begin{lemma} (\cite[Lemma 2.3]{Bas},\cite[Theorem 2.3B]{BIK1}).  Let  $C\in \mathcal N_B$. Then $C\in \mathcal U_B$ iff $C$ satisfies the following condition for all $i\in S_P$:
\begin{equation}\label{uppertrieq}
< C( (1,i,k))\,,  (1,i,k')> \,=0 \text { whenever }\, 1\le k'\le k\le n_i.
\end{equation}\noindent
Also, $\mathcal U_B$ is a maximal nilpotent \emph{subalgebra} of $\mathcal C_B$, and is isomorphic as a variety to an affine space.
\end{lemma}
\begin{proof} The condition \eqref{uppertrieq} is equivalent to the strict 
upper triangularity of $\varphi_i  (C)$.  That $\mathcal U_B$ is a maximal nilpotent subalgebra of $\mathcal N_B$ follows from \eqref{UBeq}, and the fact that each $U_T(W_i)$  is a maximal nilpotent subalgebra of $\End_{\sf k}(W_i)$. 
 It is straightforward to write coordinates for $\mathcal U_B$ as an affine space, using \eqref{commutatoreq}, and the $B$-action from \eqref{Vsumeq},\eqref{Vbasiseq}.
\end{proof}

\subsection{The poset $\mathcal D_P$.}
 We stated earlier that $\mathcal D_P$ is the poset (or digraph) associated to the algebra $\mathcal U_B\subset \mathcal C_B$. That is, the  elements (or vertices) of $\mathcal D_P$ correspond 1-1 to the basis elements of $V$ from 
\eqref{Vsumeq}, \eqref{Vbasiseq}; two elements $b,b'$ in $\mathcal D_P$ satisfy
\begin{equation}\label{dpubeq}
 b< b' \Leftrightarrow  \exists \,A\in \mathcal U_B \text { such that } <A( b),b'>\not=0.
\end{equation} We now give a second definition of $\mathcal D_P$ by specifiying its diagram $\mathrm{Diag}(\mathcal D_P)$, comprised of the pairs $b< b'$ in $\mathcal D_P$ such that $b'$ \emph{covers} $b$ (there are no vertices $x, b<x<b'$; we will also say $b$ \emph{precedes} $b'$). We determine these pairs by their corresponding \emph{elementary maps} in $\mathcal U_B$ (see below). For the equivalence of the two definitions, see \cite[Theorem 2.5, (2.18), Remark 2.10]{BIK1}. For $i\in S_P$ we denote by $i^+=\min\{s\mid s\in S_P, s>i\}$ and $i^-=\max
\{s\mid s\in S_P,s<i\}$ the next largest and next smaller elements of $S_P$, respectively, when they exist.

\begin{definition}\label{DPdef}
\cite[Def. 2.9]{BIK1}. (Maps and poset $\mathcal D_P$ associated to $P$)
\begin{enumerate}[a.]
\item\label{DPdefa} Vertices of $\mathcal D_P$.
For each pair $(u,i)$ with $i\in S_P$ and $  1\le u\le i$, there are $n_i$ vertices  $\{ (u,i,k), 1\le k \le n_i\}$.  We visualize these as a vertical column parallel to the $z$-axis in 3-space where $(u,i,1)$ as the bottom vertex and $(u,i,n_i)$ is the top vertex of the column. \par
\item Elementary maps of $\End_{\sf k}(V)$. The maps defined below are zero on those basis elements of $V$ from \eqref{Vsumeq} and \eqref{Vbasiseq} not specifically listed.
\begin{enumerate}[i.]
\item for $i\in S_P\backslash p_t$, $\beta_i=\beta_{i,i^-}$ maps the vertex $(u,i,n_i)$ to $(u,i^-,1)$, whenever $ 1\le u\le i^-$.
\item for $i\in S_P\backslash p_t$, $\alpha_{i}=\alpha_{i^-,i}$ maps $(u,i^-,n_{i^-})$ to $ (u+i-i^-,i,1),$ whenever $ 1\le u\le i^-.$
\item  $e_{i,k}$ maps the vertex $(u,i,k)$ to $(u,i,k+1), 1\le u\le i,1\le k<n_i$.
\item When $i\in S_P$ is isolated (when neither $i-1\in S_P$ nor $ i+1\in S_P$),   the map $w_i$ sends $(u,i,n_i)$ to 
$(u+1,i,1)$ whenever   $1\le u<i$.
\end{enumerate}
\item\label{DPdefb} There is an edge $v\to v'$ in the diagram $\mathrm{Diag}(\mathcal D_P)$ iff $\exists$ an elementary map $\gamma$ such that $\gamma (v)=v'$. 
\end{enumerate} 
\end{definition}
\begin{example}
$\mathcal D_P$ for $P=(4,2,2,1)$. There are four rows, three levels $i=4,2,1.$ See 
Figure \ref{posetfigure}.
\end{example}
By giving maps corresponding to the edges of the diagram of $\mathcal D_P$ we have in effect defined  a large quiver $\mathfrak Q_P$ with identities (\cite[Definition 2.9]{BIK1}).
\begin{figure}[htb]
\begin{center}
\includegraphics[scale=0.63]{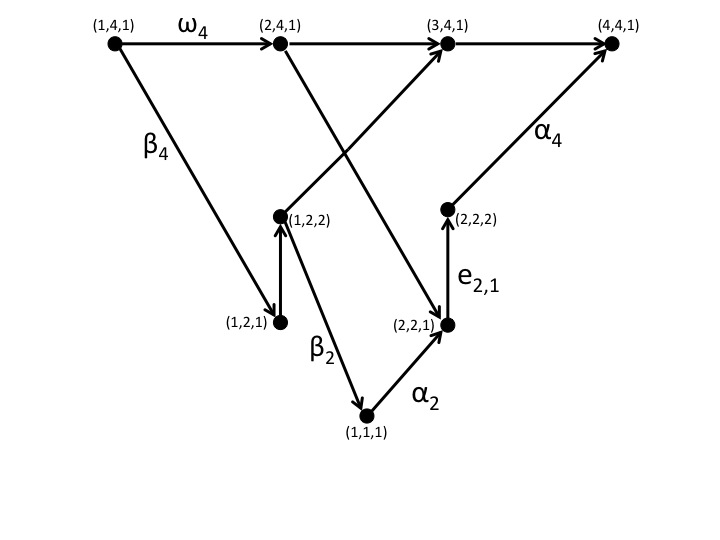}
\end{center}
\vskip -2.5cm
\caption{$\mathrm{Diag}(\mathcal D_P)$ and maps for $P=(4,2,2,1)$.}\label{posetfigure}
\end{figure}
 The $B$-action on $\mathcal D_P$ maps vertices one step to the right:
$B(u,i,k)=(u+1,i,k)$ if $u<i$ and is $0$ when $u=i$ (see \eqref{Vbasiseq}). Evidently, the elementary maps commute with the action of $B$. The $B$-orbit of an edge $v\to v'$ of the diagram is the set of edges $B^s v\to B^sv', s=1,2,\ldots $ for which $B^s\cdot v'\not=0$. Definition~\ref{DPdef}b above assigns a unique map to each maximal $B$-orbit of $\mathrm{Diag}(\mathcal D_P)$. Thus, there is a $1-1$ correspondence between the maximal $B$-orbits of edges of  $\mathrm{Diag}(\mathcal D_P)$ and the set of elementary maps. Also, if
$
(u,i,k) \text { precedes } (u',i',k')$ then either
\begin{equation*}  B( (u,i,k)) \text { precedes }B( (u',i',k')), \text { or }  B((u',i',k'))=0.
\end{equation*}
Likewise $v\le v' $ and $Bv'\not=0$ imply $Bv\le Bv'$.
 It follows from \eqref{UBeq} and Definition \ref{DPdef} that $\mathcal U_B\subset \mathfrak I (\mathcal D_P)$ is the subalgebra of the incidence algebra $\mathfrak I (\mathcal D_P)$ over $\sf k$ comprised of its  nilpotent elements satisfying $[A,B]=0$, or, equivalently, \eqref{commutatoreq}.  \par
The \emph{$i$-level} of $\mathcal D_P$ is the set of all vertices with second entry $i$: equivalently, the vertices of $\{\bigcup_{1\le k \le n_i}V_{i,k} \}$.
We denote by $\varrho (u,i,k)=\varrho (u,i)= 2u-i-1$ the integer giving the
relative position of a vertex with respect to the vertical center of symmetry of $\mathcal D_P$, determined by the involution $\tau $ of $\mathcal D_P$  (see \cite{BIK1} and \eqref{taueq} below). 
\begin{lemma} \cite[Theorem 2.13]{BIK1} Let $(u,i,k)\le (u',i',k') $ in $\mathcal D_P$. Then
\begin{align}\label{1eq}
&u\le u'\\
&i-u\ge  i'-u'\label{2eq}\\
& \varrho (u,i)+\mid i'-i\mid\,\  \le\, \varrho (u',i').\label{3eq}
\end{align}
\end{lemma}
\begin{proof} \eqref{1eq} is immediate from  \eqref{dpubeq} and $[A,B]=0$.
We have
\begin{equation*}
 B^{i+1-u}\cdot A( (u,i,k))= A\cdot B^{i+1-u}( (u,i,k))=0,
\end{equation*}
implying \eqref{2eq}.
 \eqref{3eq} follows from \eqref{1eq} and \eqref{2eq}.
\end{proof}\par
 We may write $(u,i)$ for $(u,i,1)\in \mathcal D_P$ when the multiplicity $n_i=1$. 
\begin{corollary}\label{bpowcor} Let $p$ be a chain from $v=(u,i,k)$ to $v'=(u',i',k')$ in $\mathcal D_P$. Then the total number $m(p,b)$ of
$\beta_b$ and $w_b$ edges in $p$ satisfies
\begin{equation}\label{bpoweq}
m(p,b)\le u'-u-\min\{|i-b |,| i'-b|\}
\end{equation}
if  $b<\min (i,i')$ or $b>\max (i,i')$; and it satisfies $m(p,b)\le u'-u$ otherwise.
\end{corollary}
\begin{proof} Assume $b<\min (i,i')$, and that $p$ is saturated. Then $\beta_b, \alpha_b$ edges are paired in $p$. Thus, if the chain first hits level $b$ at $(u_1,b,1)$ and last at $( u_2,b,n_b)$ we have $m(p,b)\le u_2-u_1$; from \eqref{1eq},\eqref{2eq} we have $u\le u_1, u_2+i'-b\le u' $ so $m(p.b)\le u'-u-|i'-b|$. Likewise for $b>\min(i,i')$ we have $m(p,b)\le u'-u-|i-b|$.
\end{proof}\par
 See \cite[Proposition 2.14]{BIK1} for a generallization specifying $\varrho(v')-\varrho(v)$ for $p$.
\subsection{The $U$-chains of $\mathcal D_P$.}
 
For $ S\subset S_P$ we denote by $\iota^{-1}(S)
$ the subpartition of $P$ comprised of all parts of $P$ having lengths in $S$. An \emph{ $s$-chain} of  a poset $\mathcal D$ is a union of $s$ chains of $\mathcal D$. T	he length of a chain is its number of vertices. 
The concept of $U$-chains of $\mathcal D_P$ is essentially due to
 P.~Oblak (``$B_k$ paths'' in
\cite{Obl1}, see also \cite[\S 3]{BIK1}).

\begin{definition}\label{Uchaindef} A \emph{simple $U$-chain} $U_a\subset\mathcal D_P$ is comprised of the following vertices, and edges in $\mathcal D_P$ between adjacent vertices:
\begin{enumerate}[i.]
\item all the vertices at levels $a,a-1$ of $\mathcal D_P$.
\item two \emph{hooks} above the $a$-level:\par the first is comprised of all vertices
$
(1,i,k) \mid i>a, 1\le k\le n_i;$\par
 the second is
comprised of all vertices $(i,i,k)\mid i>a, 1\le k\le n_i$.
\end{enumerate}
\end{definition}
\begin{note}\label{shellnote}
The simple chain $U_a$ in $\mathcal D_P$  is comprised of the $\mathcal D_P$ levels $a,a-1$ corresponding to an almost rectangular subpartition $P'=\iota^{-1}\{a,a-1\}$ of $P$, union the two hooks, one on the left from the source $(1,p_1,1)$ of $\mathcal D_P$ down to $P'$ and the other symmetrically located on the right from $P'$ up to the sink $(p_1,p_1,n_{p_1})$ of $\mathcal D_P$.  The length $| U_a |$ satisfies equation \eqref{obeq}. When $a$ is isolated in $S_P$ the simple $U$-chain $U_a$ in $\mathcal D_P$ is comprised of the chain at level $a$ of $\mathcal D_P$ union the two hooks. \par
 The diagram of $\mathcal D_P$ is comprised of the covering edges of $\mathcal D_P$ (Defnition \ref{DPdef}). We need an augmented diagram, whose role will become apparent after we define $s$-$U$-chains, in Lemma \ref{singletonlem}.
\begin{definition}\label{MCSdef} 
 A \emph{maximal consecutive subsequence} (MCS) of $S_P$ is one not properly contained in a larger consecutive subsequence. We denote by $S_P''$ the subset of $S_P$ comprised of minimum elements of all MCS having odd cardinality.
The \emph{augmented diagram}  
$
\mathrm{Diag}^{aug}(\mathcal D_P)\supset \mathrm{Diag}(\mathcal D_P)$
is the diagram  $\mathrm{Diag}(\mathcal D_P)$ supplemented by new edges 
$(u,\ell,n_\ell)\mapsto (u+1,\ell ,1)$
 for each pair $ (u,\ell ) $ such that $ 1\le u<\ell$ and  $\ell\in S_P''$, $\ell$ not isolated.

\end{definition}
An isolated $\ell \in S_P$ is the minimum of an MCS of length one in $S_P$: the corresponding edges are already in $ \mathrm{Diag}(\mathcal D_P)$.  
\par  Recall from \cite[Definition 2.15]{BIK1} the order reversing involution 
\begin{equation}\label{taueq}
\tau: \mathcal D_P\to \mathcal D_P,\,\, \tau (u,i,k)=(i+1-u,i,n_i+1-k).
\end{equation} 
A $U$-chain $U_a$ is evidently mapped to itself by $\tau$, the left hand hook mapping to the right hand hook.
\par
Let $C,C'$ be two disjoint $\tau$-symmetric chains of $\mathrm{Diag}^{aug}(\mathcal D_P)$, that are maximal with respect to the properties of being disjoint and symmetric. We say that $C'$ is \emph{inside} $C$ if for each row $(1,i,k)\le (2,i,k),\ldots \le (i,i,k)$ of $\mathcal D_P$ (so $i,k$ are fixed), all vertices of $U'$ in the row lie between the outside two vertices of $U$ in that row. A \emph{shelling} of a $\tau$-symmetric subset $\mathcal D$ of the vertices of $\mathcal D_P$ is a sequence of  $s$ disjoint $\tau$-symmetric chains $C_1,\ldots ,C_s$ of $\mathcal D_P$ whose union is $\mathcal D$ and such that $C_{i+1}$ is inside $C_i$ for $ i=1,\dots ,s-1$.
\end{note}
We now define $s$-$U$-chains of $\mathcal D_P$.

\begin{definition}\label{suchaindef}
{\bf A.} Let $\mathfrak A=(a_1,a_2,\ldots ,a_s)$ be an $s$-tuple of positive integers satisfying $a_i\in S_P$ and $a_i\ge a_{i+1}+2$ for $1\le i<s$. We define $\{\mathfrak A\}=\{a_1,a_1-1,a_2,a_2-1,\ldots ,a_s,a_s-1\}$.
\begin{enumerate}[a.]
\item We denote by $\{U_{\mathfrak A}\}$ the subset of vertices of $\mathcal D_P$ comprised of
\begin{enumerate}[i.]
\item all vertices in the levels of $\mathcal D_P$ given by $\iota^{-1}(\{\mathfrak A\}\cap S_P)$;
\item for each level $\ell >a_s \mid \ell \in S_P\backslash (\{\mathfrak A\}\cap S_P)$, all vertices 
\begin{align*}
&(u,\ell,k) \text { with } u\le \#\{ a_i<\ell\}, 1\le k\le n_\ell \text { (at the left of the $\ell$ level);}\\
& (u,\ell, k), \text { with } u\ge \ell+1- \#\{ a_i<\ell\}, 1\le k\le n_\ell \text { (at the right of the $\ell$ level.)}
\end{align*}
\end{enumerate}
\item We define the $s$-$U$-chain $U_{\mathfrak A}$ as the unique shelling of $\{U_{\mathfrak A}\}$ by a set of $s$ disjoint $\tau$-symmetric chains of $\mathcal D_P$. The first and outside chain in the shelling is the simple $U$-chain $U_{\mathfrak A,1}= U_{a_s}$ of $\mathcal D_P$. The $\Upsilon$-th component chain $U_{\mathfrak A,\Upsilon}$ -- counting from the outside of $\mathcal U_{\mathfrak A}$ -- has vertices 
\begin{align}\label{shelleq}
\{U_{\mathfrak A,\Upsilon}\}&=\{(u,\ell, k)\mid\ell\in \{a_{s+1-\Upsilon},a_{s+1-\Upsilon}-1\},\, \Upsilon -1\le u \le \ell+2-\Upsilon , 1\le k\le n_\ell\}\notag\\
\bigcup &\{(\Upsilon,\ell ,k), (\ell+1-\Upsilon,\ell, k) \text { with } \,\ell>a_{s+1-\Upsilon},\ell \in S_P\backslash \{\mathfrak A\}\cap S_P,1\le k\le n_\ell\}.
\end{align}

\end{enumerate}
{\bf B.} We denote by $|U_{\mathfrak A}|$ and $|U_{\mathfrak A,\Upsilon}|$ the lengths of the $s$-chain, and of the $\Upsilon$-th component chain, respectively. We denote by  $v_{\mathfrak A ,\Upsilon,j}, 1\le j\le |U_{\mathfrak A ,\Upsilon}|$ the $j$-th vertex of the chain $U_{\mathfrak A,\Upsilon}$: so its initial vertex is $v_\Upsilon=v_{\mathfrak A,\Upsilon,1}$.
Given $U_{\mathfrak A}$ and $v_{\mathfrak A ,\Upsilon,j}$ we will term the portion of  $U_{\mathfrak A,\Upsilon} $ from $v_\Upsilon$ to $v_{\mathfrak A ,\Upsilon,j}$ the \emph{standard} chain from $v_\Upsilon$ to $v_{\mathfrak A ,\Upsilon,j}$. \par\noindent
{\bf C}. We denote by $\langle U_{\mathfrak A} \rangle $ and $\langle U_{\mathfrak A,\Upsilon}\rangle$ the $\sf k$ span of the elements of $\sf B$ (vertices of $\mathcal D_P$) in $U_{\mathfrak A}  $ and $ U_{\mathfrak A,\Upsilon}$, respectively, and by  $\langle U_{\mathfrak A} \rangle _L $ and $\langle U_{\mathfrak A,\Upsilon}\rangle _L$ the  $L$ spans of the same elements when $L\supset \sf k$ is a field. \par\noindent
{\bf D}. We say that the $s$-$U$-chain is \emph{maximal} if it is not a proper subset of another $s$-$U$-chain (with the same $s$).
\end{definition}
\begin{remark}\label{shellrem}
 The chain $U_{\mathfrak A,\Upsilon}$ is made up of vertices in what is left of the set $\{U_{\mathfrak A}\}$, after removal of the previous $\Upsilon -1$ chains. It has an almost rectangular part as in the first line of equation \eqref{shelleq}; it has as well the two outside hooks of what is left above the $a_{s+1-\Upsilon}$ level of $\mathcal D_P$ (second line of \eqref{shelleq}. The chain $U_{\mathfrak A,\Upsilon}$ has initial vertex $v_\Upsilon =v_{\mathfrak A,\Upsilon,1}=(\Upsilon,p_1,1)$ and terminal vertex $\tau (v_\Upsilon)=v_{\mathfrak A,\Upsilon,|U_{\mathfrak A,\Upsilon}|}=(p_1+1-\Upsilon ,p_1,n_1)$.\par\noindent
\end{remark}
\begin{definition}\label{singlevdef}
We say that $U_{\mathfrak A ,\Upsilon}$ has a \emph{singleton level} if  $a_{s+1-\Upsilon}-1\notin S_P$: so its almost rectangular portion has only one level.
\end{definition}

 We will need the following characterization of the levels $\ell \in S_P$ that may occur as singleton levels in a $s$-$U$-chain of $\mathcal D_P$.
\begin{lemma}\label{singletonlem} Let $ U_{\mathfrak A}$ be a maximal $s$-$U$-chain.  If
$ U_{\mathfrak A,\Upsilon}$ has a singleton level then
$a_{s+1-\Upsilon}$ is the minimum of an odd length MCS of $ S_P$ included in $ \{ {\mathfrak A}\}$. Conversely, let $\ell$ be the minimum of a length $(2k+1)$ MCS of $S_P$: then the $(k+1)$-$U$-chain $U_{\mathfrak A}$ where $ \mathfrak A=(\ell+2k,\ell+2k-2,\ldots, \ell)$  has the singleton level $\ell$.
\end{lemma}
 \begin{definition}\label{lambdaUdef} We define the partition $\lambda_U(\mathcal D_P)$ from the $s$-$U$-chains of $\mathcal D_P$. For $1\le i\le r_P$ the $i$-th part of $\lambda_U(\mathcal D_P)$ is 
\begin{align}
(\lambda_U(\mathcal D_P))_i&= u_i(\mathcal D_P)-u_{i-1}(\mathcal D_P) \text { where } u_0(\mathcal D_P)=0 \text { and  for $i > 0$}
\notag\\
u_i(\mathcal D_P)&=\max \{\mid U_{\mathfrak A}\mid {\text{ such that}}\,
\, \mathfrak A \text { is an $i$-$U$-chain in } \mathcal D_P\} .
\end{align}
Although the component chains of the $s$-$U$-chains are disjoint, this  is otherwise analogous to the definition of $\lambda ({\mathcal D})$ in \eqref{lambdaDeq} from the sets of all chains of $\mathcal D_P$. 
\begin{example}\label{4221aex} The poset $\mathcal D_P$ for $P=(4,2,2,1)$ has $t=4$ rows, $\# S_P=3$ levels of which $\ell=4\in S_P''$ is isolated. The source is $(1,4)$ the sink is $(4,4)$.  The two simple $U$-chains of $\mathcal D_P$ are (see Figure \ref{posetfigure})
\begin{align*}
&(1,4)\le (2,4)\le (3,4)\le (4,4),\text {  and }\\
& (1,4)\le (1,2,1)\le (1,2,2)\le (1,1)\le ( 2,2,1)\le (2,2,2)\le (4,4). 
\end{align*} 
The $2$-$U$-chain $U_{\mathfrak U},\mathfrak U=(4,2)$ has a singleton level $\ell =4$.  Thus we have $\lambda_U(P)=(7,2)$, the first difference of $(u_0=0,u_1=7,u_2=9$).
\end{example}

\begin{example}\label{2.5ex} For $P=(5,4,3,3,2,1)$ the simple $U$-chains are $U_5, U_4,U_3,U_2$ of lengths $9,12,12,11$, respectively, according to \eqref{obeq}.  The $2$-$U$-chain $U_{(4,2)}$ (Figure \ref{mufigure}) has length 17 and is comprised of an outer chain
\small 
\begin{equation*}
U_{(4,2),1}= (1,5), (1,4),(1,3,1),(1,3,2),(1,2),(1,1),(2,2),(3,3,1),
(3,3,2),(4,4),(5,5).
\end{equation*}
\normalsize and the inner chain 
\small
\begin{equation*} 
U_{(4,2),2}=(2,5),(2,4),(2,3,1),(2,3,2),(3,4),(4,5).
\end{equation*}
\normalsize 
The other maximal $2$-$U$-chains are $U_{5,3}$ and $U_{5,2}$ of lengths $17$ 
and 16, respectively.  The unique $3$-$U$-chain $U_{(5,3,1)}$ has a singleton level $\ell=1$; it has the shelling \small
\begin{align*}
U_{(5,3,1),1}&=(1,5),(1,4),(1,3,1),(1,3,2),(1,2),(1,1),(2,2),(3,3,1),(3,3,2),(4,4),(5,5)\\
U_{(5,3,1),2}&=(2,5),(2,4),(2,3,1),(2,3,2),(3,4),(4,5),\text { and }\\
U_{(5,3,1),3}&=(3,5).
\end{align*}
\normalsize
Thus, $v_{(5,3,1),2,3}=(2,3,1)$, the third vertex of $U_{(5,3,1),2}$; and $v_{(5,3,1),2,6}=(4,5)$.\par
The partition $\lambda_U(P)=(12,5,1)$, the first differences of $ (u_0=0,u_1=12,u_2=17,u_3=18)$. Note that neither of the maximum-length simple $U$-chains $U_4,U_3$ is the first component $U_{(4,2),1}$ or $U_{(5,3),1}$ of a maximum-length $2$-$U$-chain!
\end{example}

\end{definition}
\subsection{The homomorphism $A_{\sf R}$.}\label{Asec}
We first define a polynomial ring $\sf R$ over $\sf k$, most of whose variables correspond $1-1$  to the maximal $B$-orbits of edges in the diagram of $\mathcal D_P$; then we will define a certain sparse matrix $A_{\sf R}\in \mathcal U_{B.\sf R}=\mathcal U_B\otimes_{\sf k}\sf R$.
We let
\begin{equation}\label{A0eq}
{\sf R= {k}}[ s _i,t_i,t_{j,k},z_\ell \mid i\in S_P\backslash p_t, j\in S_P,1\le k\le n_i, \ell \in {S_P}'' ].
\end{equation}

 Let
$\sf F$ be the quotient field of $\sf R$, $V_R=V\otimes _{\sf k}{\sf R},\,V_{\sf F}=V\otimes _{\sf k}{\sf F}$. We identify $\End_{\sf k}V$ with $\Mat_n{\sf k}$, $\End_{\sf R}V_{\sf R}$ with $\Mat_n{\sf R}$ and $\End_{\sf F}V_{\sf F}$ with $\Mat_n{\sf F}$ in the basis $\sf B$.
\begin{definition}\label{simplyadequatedef} We define the \emph{simply adequate} matrix $A_{\sf R}\in 
\End_{\sf R} V_{\sf R}=\Mat_n{\sf R}$ as
\begin{equation}\label{Aeq}
A_{\sf R}=\sum_{i\in S_P\backslash p_t}( s_i\beta_{i}+t_i\alpha_{i})+{\sum}'t_{i,k}e_{i,k}+{\sum_{\ell\in {S_P}''}}z_{\ell} w_{\ell}
\end{equation}
where $\sum'$ is the sum over couples $(i,k)$ with $ 1\le k<n_i, i\in S_P$. Here $\beta_i,\alpha_i$ and $ e_{ik} $ are the elementary endomorphisms of $V$ given in Definition \ref{DPdef}b; and $w_\ell$ is the endomorphism of $V$ taking $(u,\ell,n_\ell)$ to $(u+1,\ell ,1)$ for $ 1\le u<\ell$, which is elementary only when the MCS containing $\ell\in S_P'' $ is a singleton. 
\end{definition}
Equivalently, we have the following description of the entries of the matrix $A_{\sf R}$.
\begin{equation} 
(A_{\sf R})_{v,v'}=
\begin{cases} \!\!&\text { the variable of $\sf R$ determined by the map  $v\to v'$ when $v$ precedes $v'$};\\\!\!
&\text { the variable $z_i$ when $v=(u,\ell,n_\ell) $ and $v'=(u+1,\ell,1)$ and $\ell \in {S_P}''$};\\\!\!
& \,0 \text { otherwise}.
\end{cases}
\end{equation}
In particular the variables $z_\ell$ of the simply adequate $A_{\sf R}$ of \eqref{Aeq} correspond 1-1 to the  singleton levels in maximal $s$-$U$-chains $\mathfrak A$ of $\mathcal D_P$ (Lemma \ref{singletonlem}).
\begin{definition}\label{adequatedef}
 Let $L$ be a field containing $\sf k$. We call $A\in \mathcal U_B\otimes_{\sf k}L$ \emph{adequate} if there exist $s_i,t_i,t_{i,k}, z_\ell \in L\backslash 0$ for every $i\in S_P\backslash p_t$, every $k\in \{1,\ldots ,n_i\}$, and every $\ell \in S_P''$, such that, in the notation of \eqref{Aeq},
\begin{equation}\label{Aadeqeq}
 A=\sum_{i\in S_P\backslash p_t}( s_i\beta_{i}+t_i\alpha_{i})+{\sum}'t_{i,k}e_{i,k}+{\sum_{\ell\in {S_P}''}}z_{\ell} w_{\ell}.
\end{equation} 

\end{definition}
 In \cite{BIK1} we conjectured that if $A$ is adequate, then $P_A=Q(P)=\Ob (P)$. We will show the weaker result that if $\sf k$ is an infinite field and $A_{\sf R}$ is simply adequate then $P_{A_{\sf R}}\ge \lambda_U(\mathcal D_P)$, where $A_{\sf R}$ is considered as an element of $\Mat_{\sf F}(n)$ (Corollary~\ref{maincor}). We then show that there exists an adequate $A$ over $\sf k$ such that $P_A\ge \lambda_U(\mathcal D_P)$ (Theorem~\ref{mainthm}). The need for a hypothesis such as ``adequate'' is shown by
\cite[Example 3.17c]{BIK1}.\par
\begin{note}\label{saturatednote}
A chain in  $\mathrm{Diag}(\mathcal D_P)$ or  $\mathrm{Diag}^{aug}(\mathcal D_P)$ from vertex $v$ to vertex $v'$ is \emph{saturated} if it is not a proper subset of another chain from $v$ to $v'$ in $\mathrm{Diag}(\mathcal D_P)$ or  $\mathrm{Diag}^{aug}(\mathcal D_P)$, respectively. 
For $u\in \mathbb N$ the entry $(A_{\sf R}^u)      _{v,v'}$ of $A_{\sf R}^u$ is the projection $<A_{\sf R}^u( v), v'>$. It is a sum of terms, most of which are  monomials in $\sf R$ corresponding to a saturated chain in $\mathrm{Diag}(\mathcal D_P)$ from $v$ to $v'$. However, we have included extra variables
$z_\ell$, each corresponding to a map $w_\ell$ and to the $B$ orbit of an edge $(1,\ell, n_\ell)\to (2,\ell, 1)$ in $\mathrm{Diag}^{aug}(\mathcal D_P)\subset \mathcal D_P$, where $\ell\in S_P'',\,\ell$ not isolated. Thus, $(A_{\sf R}^u)      _{v,v'}$ includes monomials corresponding to chains from $v$ to $v'$ in $\mathrm{Diag}^{aug}(\mathcal D_P) $ (Lemma \ref{Apowerlemma}).
 We chose the simply adequate $A_{\sf R}$ -- a relatively sparse matrix -- in order to simplify a key step in our proof (see \eqref{pows1eq}ff of Proposition~\ref{2lem}): $A_{\sf R}$ has the mininum number of variables that we need for this step.  We could have worked directly with a generic $A'=A_{\sf R'}\in \mathcal U_{B,\sf R'}$ over a large ring $\mathsf R'$: for $v<v'\mid v,v'\in \mathcal D_P$ the entry $A'_{v,v'}$ is a variable of $\sf R'$ corresponding to the maximal $B$-orbit containing the interval $[v,v']$ in $\mathcal D_P$. Using the sparse matrix $A_{\sf R}$ leads to a more precise statement.
\end{note}

\subsection{The projection $\pi_{\mathfrak A}:\mathcal T_{\mathfrak A}\to\mathcal U_{\mathfrak A}$, and the matrix $M_{\mathfrak A}$.} We fix $A_{\sf R} \in \mathcal U_{B,\sf R}$ to be the simply adequate matrix of Definition \ref{simplyadequatedef}.   Let $U_{\mathfrak A}$ be an $s$-$U$--chain of $\mathcal D_P$. Recall that the initial vertices $v_\Upsilon$ of the $\Upsilon$-component chain $U_{\mathfrak A,\Upsilon}$ of $U_{\mathfrak A}$ satisfy $v_\Upsilon=v_{\mathfrak A,\Upsilon,1}=(\Upsilon,p_1,1)\in V, 1\le \Upsilon\le s.$  The matrix ${A_{\sf R}}^u$ has entries $({A_{\sf R}}^u)_{b,b'}$ on each pair $b,b'\in \sf B$ of basis vectors.
\begin{definition}\label{mudef} We associate to a chain $p$ in the augmented diagram $\mathrm{Diag}^{aug}(\mathcal D_P)$ the monomial $\mu_p$ obtained by multiplying the variables of $\sf R$ in \eqref{Aeq} that are  the coefficients for the elementary maps of \eqref{DPdef} and also those variables $z_\ell$ with $ \ell \in S_P"$ corresponding to the edges of  $p$. 
\end{definition}
\begin{lemma}\label{Apowerlemma} For $v<v'$ vertices of $\mathcal D_P$, the entry ${(A_{\sf R}}^u)_{v,v'}$ of the $u$-th power ${A_{\sf R}}^u$ is the sum of degree-$u$ monomials in $\sf R$,
\begin{equation}\label{sumchaineq}
({A_{\sf R}}^u)_{v,v'}={\sum_p}' \mu_p
\end{equation}
where the sum is over all chains $p$ of length $u+1$ from $v$ to $v'$ in $\mathrm{Diag}^{aug}(\mathcal D_P)$. 
\end{lemma}
\begin{proof} This is a standard result concerning the incidence algebra of
a poset. 
\end{proof}
\begin{example}\label{4221bex} Set $P=(4,2,2,1)$. (See Figure \ref{posetfigure} and Example~\ref{4221aex}.) The chain  $p: (1,4)\to (2,4)\to (3,4)\to (4,4)$ in $\mathcal D_P$ contributes the monomial $\mu_p={z_4}^3\in \sf R$ to the entry $({A_{\sf R}}^3)_{(1,4),(4,4)}$. The chain\par
$p'=(1,4)\to (1,2,1)\to (1,2,2)\to (1,1)\to ( 2,2,1)\to (2,2,2)\to (4,4) $ contributes the monomial $\mu_{p'}=\beta_4\cdot {e_{2,1}}^2\cdot \beta_2\cdot \alpha_2\cdot \alpha_4\in \sf R$ to the entry $({A_{\sf R}}^6)_{(1,4),(4,4)}$.
\end{example}
\begin{definition}[Projection $\pi_{\mathfrak A}$ from $\mathcal T_{\mathfrak A}$ to $\mathcal U_{\mathfrak A}$]\label{pidef}
\begin{enumerate}[A.]\item Denote by ${\sf F}[x]$ the polynomial ring in one variable. Let $U_{\mathfrak A} $ be an $s$-$U$-chain. For every $u,\Upsilon\in \mathbb Z$ with 
\begin{equation}\label{Teq}
  1\le \Upsilon\le s \text { and } 0\le u\le ( \mid U_{\mathfrak A,\Upsilon}\mid -1)
\end{equation} we put
\begin{equation}
T(U_{\mathfrak A})_{u,\Upsilon}=x^u\otimes_{\sf F}v_{\mathfrak A,\Upsilon,1}\in {\sf F}[x]\otimes_{\sf F}V_{\sf R}.
\end{equation}
We set 
\begin{equation}
T(U_{\mathfrak A})=\{ T(U_{\mathfrak A})_{u,\Upsilon}\mid 1\le \Upsilon\le s,\, 0\le u\le ( \mid U_{\mathfrak A,\Upsilon}\mid -1)\}.
\end{equation}
Denote by $\mathcal T_{\mathfrak A}=\langle T(U_{\mathfrak A})\rangle \subset {\sf F}[x] \otimes_{\sf F}V_{\sf F}$  and $\mathcal U_{\mathfrak A}=\langle\{ U_{\mathfrak A}\}\rangle \subset V_{\sf F} $ the respective $\sf F$-linear spans. 
\item\label{Mdef}
There is a natural homomorphism $\omega:{\sf F}[x]\otimes_{\sf F} V_{\sf F}\to V_{\sf F}$ 
\begin{equation}
 \omega (x^s\otimes_{\sf F}v)= A^s( v),
\end{equation} 
and, since $\{U_{\mathfrak A}\}$ is  a subset of the basis $\sf 
B$ for $V_{\sf F}$, a natural projection $\rho$ from $V_{\sf F}$ to the subspace $\mathcal U_{\mathfrak A}$. We denote by $\pi_{\mathfrak A}:\mathcal T_{\mathfrak A}\to \mathcal U_{\mathfrak A}$ the composition $\rho\circ\omega$. To define the matrix $M_{\mathfrak A}$ of $\pi_{\mathfrak A}$ we simply order the set  $\{U_{\mathfrak A}\}$ by  $v_{\mathfrak A,\Upsilon,j}<' v_{\mathfrak A,\Upsilon',j'}$ if $ \Upsilon< \Upsilon'$ or $\Upsilon=\Upsilon'$ and $j< j'$. We similarly order the set $T(U_{\mathfrak A})$ by setting $x^u\otimes_{\sf F} v_{\mathfrak A,\Upsilon,1}<' x^{u'}\otimes v_{\mathfrak A,\Upsilon',1}$ if $\Upsilon< \Upsilon'$  or $\Upsilon=\Upsilon'$ and $u< u'$. We denote by $M_{\mathfrak A}$ the $|U_{\mathfrak A}|\times |U_{\mathfrak A}|$ matrix of $\pi_{\mathfrak A}$ with respect to these ordered bases. That is, the entry of $M_{\mathfrak A}$ in the $x^u\otimes_{\sf F} v_{\mathfrak A,\Upsilon,1}$  row and the $v_{\mathfrak A,j,u'}$ column, with $1\le \Upsilon,j\le s$ and $ 0\le u< \mid  U_{\mathfrak A,\Upsilon}\mid,1\le u'< \mid  U_{\mathfrak A,j}\mid $ is
\begin{equation}\label{Mentryeq}
<A^u( v_{\mathfrak A,\Upsilon,1}), v_{\mathfrak A,j,u'}>.
\end{equation}
  (See Figure \ref{421fig} and Example \ref{421advex} for $M_{\mathfrak A}$ when $P=(4,2,1)$ and $\mathfrak A=(4,2)$.)
\item
We define the \emph{standard chain} 
\begin{equation}\label{standardchain}
p_{\mathfrak A,\Upsilon ,j}:\, v_{\mathfrak A,\Upsilon ,1}\to v_{\mathfrak A,\Upsilon ,2}\to \cdots \to v_{\mathfrak A,\Upsilon,j}
\end{equation} 
in $U_{\mathfrak A,\Upsilon}$ from the initial vertex $v_{\mathfrak A,\Upsilon,1}$ to $v_{\mathfrak A,\Upsilon,j}$.  We denote by $\mu_{\mathfrak A,\Upsilon,j}\in \sf R$ the monomial of degree $j-1$ in $\sf R$ arising as in Lemma~\ref{Apowerlemma}, from this standard chain.  We denote by $\mu_{\mathfrak A,\Upsilon}$ the monomial
\begin{equation}
\mu_{\mathfrak A,\Upsilon}=\prod_{1\le j\le \mid U_{\mathfrak A,i}\mid} \mu_{\mathfrak A,\Upsilon,j}.
\end{equation}
 The \emph{distinguished} monomial of $\det M_{\mathfrak A}$ for the $s$-$U$-chain $\mathfrak A$  is the product
\begin{equation}\label{mueq}
\mu_{\mathfrak A}=\prod_{1\le \Upsilon\le s} \mu_{\mathfrak A,\Upsilon}.
\end{equation}
\end{enumerate}
\end{definition}
\noindent
\begin{note}\label{Mentrynote} Evidently, the dimensions of the vector spaces $\mathcal T_{\mathfrak A}$ and $\mathcal U_{\mathfrak A}$ are the same. The degree of $\mu_{\mathfrak A,\Upsilon}$ satisfies
\begin{equation*}
\deg \mu_{\mathfrak A,\Upsilon}=\left( 1+2+\cdots +(\mid U_{\mathfrak A,\Upsilon}\mid -1)\right)= (\mid U_{\mathfrak A,\Upsilon}\mid -1)(\mid U_{\mathfrak A,\Upsilon}\mid)/2
\end{equation*}
The entry $<A^u( v_{\mathfrak A,\Upsilon,1}), v_{\mathfrak A,j,u'}>$ of $M_{\mathfrak A}$ is the sum of the monomials $\mu_p$ of $\sf R$ corresponding as in \eqref{sumchaineq} to length-$(u+1)$ chains $p$ from $v_{\mathfrak A,\Upsilon,1}$ to  $v_{\mathfrak A,j,u'}$ in $\mathrm{Diag}^{aug}(\mathcal D_P)$.\par
The distinguished monomial $\mu_{\mathfrak A}$ occurs in the main diagonal term of $\det M_{\mathfrak A}$.
For $\mathfrak A'\subset \mathfrak A$, $M_{\mathfrak A'}$ is a principal submatrix of $M_{\mathfrak A}$. For example, when $P=(4,2,1)$, $\mathfrak A'=(4), \mathfrak A=(4,2)$, 
 $M_{\mathfrak A'}$ is the leading $5\times 5$ principal submatrix of $M_{\mathfrak A}$ (Example~\ref{421advex}).
\end{note}
\section{Lower bound  for $Q(P)$.}\label{3sec}
The key steps in the proof of Theorem \ref{mainthm} involve an analysis of the sets of chains from the initial vertices of the $s$-$U$-chains to all the vertices of $U_\mathfrak A$. Each such set leads to a factorization of a monomial $\nu\in \sf R$ occuring in the expansion of the determinant $\det (M_{\mathfrak A})$. Using the sparseness of $A_{\sf R}$ -- that simplifies our work -- we show that there is a unique such factorization leading to the monomial $\mu_{\mathfrak A}$ of \eqref{mueq}  (Proposition~\ref{2lem} for $2$-chains and Theorem \ref{sthm} for $s$-chains). This shows that the Jordan block partition $P_{A_{\sf R}}$ dominates $ \lambda_U(\mathcal D_P)$ (Corollary \ref{maincor}).\par
The following result is well known. 
\begin{lemma}\label{ranklem} Let $V$ be an $n$-dimensional vector space over a field ${\sf F}$ and let $A$ be a nilpotent matrix in $\Mat_n({\sf F})=\End_{\sf F}V_{\sf F} $. The Jordan type 
$
Q=P_A=(q_1,\dots q_r)$, $ q_1\ge \cdots \ge q_r 
$ of $A$
satisfies
\begin{equation}\label{rankeq}
\forall i \in \{1,\ldots ,r\}, \quad\sum_{k=1}^i q_k=\max \{ \dim_{\sf F}\langle {\sf F}[A]\circ\{ v_1,\cdots v_i\}  \mid
v_1,\ldots ,v_i\in V.
\end{equation}
\end{lemma}\noindent
\begin{proof} By the action of $A$ as $X$,  $V$ is a finitely generated torsion ${\sf k}[X]$-module, the direct sum of cyclic modules $V=\oplus_{k=1}^r{\sf F}[A]\circ v_i\cong \oplus_{k=1}^r
{\sf F}[X]/(X^{q_k})$ whose lengths correspond to the Jordan type of $A$. 
This provides a set of cyclic vectors $z_1,\ldots ,z_r$ satisfying, for each $i, 1\le i\le r, \dim_{\sf F}\langle {\sf F}[A]\circ \{z_1,\ldots ,z_i\}\rangle=\sum_{k=1}^i q_k$. That $\sum_{k=1}^i q_k$ is the maximum dimension of a subspace generated by $i$ vectors is a consequence of the uniqueness of the Jordan partition. 
\end{proof}

We now prepare to show that the monomial $\mu_{\mathfrak A}$  occurs only once in the expansion of  $\det M_{\mathfrak A}$, where $M_{\mathfrak A}$ is the $|U_{\mathfrak A}|\times |U_{\mathfrak A}|$ matrix of Definition \ref{pidef}\ref{Mdef}. There is a natural bijection  $\eta$ from the set of rows to the set of columns of $M_{\mathfrak A}$:
\begin{equation}
\eta : T(U_{\mathfrak A})\to \{U_{\mathfrak A}\} ;\, \eta (x^u\otimes_{\sf F} v_{\mathfrak A,\Upsilon ,1})= v_{\mathfrak A,\Upsilon ,u+1},  0\le u< |U_{\mathfrak A,\Upsilon}|, 1\le \Upsilon\le s.
\end{equation}\par
Here $\det M_{\mathfrak A}$ is the sum of $|U_{\mathfrak A}|!$ terms, one for each permutation $\sigma$ of $\{U_{\mathfrak A}\}$. The term corresponding to $\sigma $ is
\begin{equation}\label{detMeq} 
{\text{sgn}} (\sigma )\prod_{1\le \Upsilon\le s,\,0\le u< \mid U_{\mathfrak A,i} \mid}<A^u( v_{\mathfrak A,\Upsilon ,1}), \sigma (v_{\mathfrak A,\Upsilon ,u+1})>,
\end{equation} 
where the sign is that of $\sigma$.  Indeed, the entry in row $(u,i)$ and column $\sigma\circ \eta (u,i)$ of $M_{\mathfrak A}$ is the sum of monomials, one for each chain $c_{\Upsilon ,u}$ of length $u+1$ from $v_{\mathfrak A,\Upsilon ,1}$ to $\sigma(v_{\mathfrak A,\Upsilon ,u+1})$ in $\mathrm{Diag}^{aug}(\mathcal D_P)$. Consequently, the term of $\det M_{\mathfrak A}$ corresponding to $\sigma$ is the sum of signed monomials $\text{sgn}(\sigma)\nu$, one for each array $C_f$ of chains as in (ii) of Definition \ref{factornotation}.
\begin{definition}\label{factornotation} Let $U_{\mathfrak A}$ be an $s$-chain. A \emph{chain factorization} $f$ of a signed monomial $\pm\nu\in 
\sf R$ in the expansion of $\det M_{\mathfrak A}$, is a triple $f=(\nu_f, \sigma_f, C_f) $ where $ C_f=\{c_{\Upsilon ,u,f}\}$ is an array of chains, and $ \nu_f=\{\nu_{\Upsilon ,u,f}\}$ is an array of monomials, comprised of
\begin{enumerate}[(i)]
\item A choice of a permutation  $\sigma_f$ of  $\{ U_{\mathfrak A}\}$. This determines the map 
\begin{equation*}
\sigma_f\circ \eta: T(U_{\mathfrak A})\to \{ U_{\mathfrak A}\}.
\end{equation*}
\item  $C_f$: For each pair $(\Upsilon ,u), 1\le \Upsilon\le s,0\le u< |U_{\mathfrak A,\Upsilon }|,$ the choice of a chain $c_{\Upsilon ,u,f}$ of length $u+1$ from $v_{\mathfrak A,\Upsilon ,1}$ to $\sigma_f (v_{\mathfrak A,\Upsilon ,u+1})$ in $\mathrm{Diag}^{aug}(\mathcal D_P)$.
\item $\nu_f$: the array of monomials $\nu_{\Upsilon ,u,f}=\mu_{c_{\Upsilon ,u,f}}$, each the product of variables of $\sf R$ corresponding to the edges of $c_{\Upsilon ,u,f}$ (Definition \ref{mudef}).
 \end{enumerate}
\end{definition}
When $f=(\nu_f, \sigma_f, C_f)$ is a chain factorization of $\pm \nu$, then
\begin{equation}
\mid \nu\mid\, =\prod_{1\le \Upsilon\le s,\,\,0\le u<\mid U_{\mathfrak A,i}\mid} \nu_{\Upsilon ,u,f},
\end{equation}
and 
$\text{sgn} (\sigma_f )\cdot\nu$ is a signed monomial of $\sf R$ in the expansion of $\det (M_{\mathfrak A})$, before any cancellation.

We say that $C_f$ is a complete set of chains for $\mathfrak A$, and that  $f=(\nu_f,\sigma_f ,C_f)$ \emph{encodes} the chains $C_f$. We may omit subscripts on $\nu_f,\sigma_f,C_f$ when $f$ is clear.

\begin{note}\label{3.4note}
A complete set $C_f$ of chains for $\mathfrak A$ includes one chain to each vertex of $\{U_{\mathfrak A}\}$, but the chains may include vertices outside of $\{ U_{\mathfrak A}\}$: see Example~\ref{3.4ex}.\par
 Among the chain
factorizations is $g_{\mathfrak A}=(\nu_{\mathfrak A},e,C_{\mathfrak A})$ of the distinguished monomial $\mu_{\mathfrak A}$, given in Definition \ref{pidef}C,  where $\sigma= e$, the identity permutation, and $C_{\mathfrak A}$ is comprised of the standard chains as in \eqref{standardchain} from the initial vertex of each component chain of $U_{\mathfrak A}$ to every vertex of that chain.  \par
In principle, a monomial term  $\nu$ in the expansion of the determinant
$\det (M_{\mathfrak A})$ may equal $\mu_{\mathfrak A}$ even though the component chains encoded by the factorization of $\nu$  do not lie in $U_{\mathfrak A}$. This is so as there may be occurences of an $\alpha_k$ or $\beta_k$ at the same $ $ level as an edge
of $U_{\mathfrak A}$, but coming from an edge not in $U_{\mathfrak A}$ -- a result of the Toeplitz condition that $A_{\sf R}$ commutes with $B$.

\end{note}
\begin{example}\label{3.4ex}[Standard chains and monomials of $\det M_{\mathfrak A}$]. Let $P=(5,4,3,3,2,1)$ and $\mathfrak A=(4,2)$ (Figure \ref{mufigure}). Then $\mu_{\mathfrak A}= \mu_{\mathfrak A,1}\cdot \mu_{\mathfrak A,2}$, 
\begin{align*}
\mu_{\mathfrak A,1}&=s_5^{10}s_4^9t_{31}^{\,8}s_3^7s_2^6t_2^5t_3^4t_{31}^{\,3}t_4^2t_5\\
\mu_{\mathfrak A,2}&=s_5^5s_4^4t_{31}^{\, 3}t_4^2t_5.
\end{align*}
\begin{figure}[htb]
\begin{center}
\includegraphics[scale=.53]{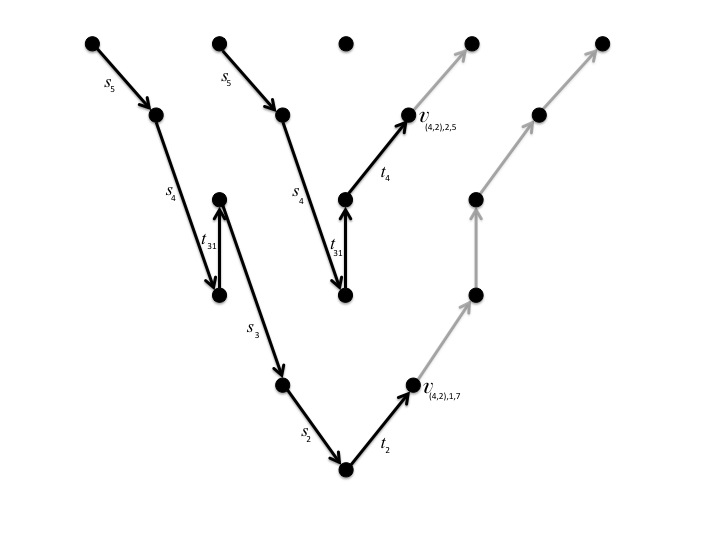}
\end{center}\vskip -2cm
\caption{$2$-$U$-Chain $U_{4,2}$ for $P=(5,4,3,3,2,1)$.} \label{mufigure}
\end{figure}
Recall from Example \ref{2.5ex} that $v_{\mathfrak A,1,7}=(2,2,1)=(2,2)$ and $v_{\mathfrak A,2,5}=(3,4)$ (see Figure \ref{mufigure}). The standard chains to these vertices are
\begin{align*} 
&(1,5)\to (1,4)\to (1,3,1)\to (1,3,2) \to (1,2)\to (1,1)\to (2,2)\\
& (2,5)\to (2,4)\to (2,3,1)\to (2,3,2)\to (3,4),
\end{align*}
respectively, corresponding to factors $\mu_{\mathfrak A,1,7}=s_5s_4t_{31}s_3s_2t_2$ and $\mu_{\mathfrak A,2,5}=s_5s_4t_{31}t_4$, respectively of the distinguished monomial $\mu_{\mathfrak A}$.
Let $\sigma = ( v_{{\mathfrak A},1,7},v_{{\mathfrak A},2,5})$, the transposition taking $(2,2)$ to $(3,4)$. There is a unique length $5$ chain $c$ from $v_{{\mathfrak A},2,1}=(2,5) $ to $(2,2)=\sigma (v_{{\mathfrak A},2,5}$), 
\begin{equation*}
c=(2,5)\to (2,4)\to (2,3,1)\to (2,3,2)\to (2,2),
\end{equation*}
an encoding (factorization) of the monomial $p_c=s_5s_4t_{3,1}s_3$. There are two length 7 chains from $v_{\mathfrak A,1,1}=(1,5) $ to $(3,4)=\sigma (v_{\mathfrak A,1,7})$, namely,
\begin{align*}
c_1=&(1,5)\to (1,4)\to (1,3,1)\to (1,3,2)\to (2,4)\to (3,5)\to (3,4) \text { and }\\
c_2=&(1,5)\to (1,4)\to (2,5)\to (2,4)\to (2,3,1)\to (2,3,2)\to (3,4),
\end{align*}
each encoding the monomial $p_{c_1}=s_5^2s_4t_{3,1}t_4t_5$. Let $\mu$ denote the monomial
\begin{equation*}
\mu= \mu_{\mathfrak A}\cdot (\mu_{\mathfrak A,1,7})^{-1}\cdot (\mu_{\mathfrak A,2,5})^{-1}\cdot (s_5s_4t_{3,1}s_3)\cdot (s_5^2s_4t_{3,1}t_4t_5)
\end{equation*}
obtained from $\mu_{\mathfrak A}$ by replacing the monomials
 $\mu_{\mathfrak A,1,7}$ and $\mu_{\mathfrak A,2,5}$  for the standard chains in $U_{\mathfrak A}$  to $(2,2)$ and $(3,4)$, by $p_{c}$ and $p_{ c_1}$, respectively: let $C_1,C_2$, respectively, denote the corresponding complete sets of chains.
Then $-\mu$ occurs twice in the $\sigma$ term of the expansion of $\det (M_{\mathfrak A})$. Once for the factorization
$f_1=(\nu_1,\sigma,C_1) $ and once for the factorization $f_2=(\nu_2,\sigma, C_2)$ where for $k=\{1,2\}$, $C_k$ is the array with $c_{14}=c$ and $c_{16}=c_k$ and all the other $c_{\Upsilon ,u,f_i}$ are the standard chains from $v_{\mathfrak A,\Upsilon , 1}$ to $\sigma (v_{\mathfrak A,\Upsilon , u+1})$ of length $u+1$.
\par  
Note that there is a second length $8$ chain $c'$ from $v_{\mathfrak A,1,1}=(1,5)$ to the vertex $v_{\mathfrak A,1,8}= (3,3,1)$ besides the standard one, namely 
\begin{equation*}
c'=(1,5) \to (1,4)\to (1,3,1)\to (1,3,2)\to (2,4)\to (3,5)\to (3,4)\to (3,3,1),
\end{equation*}
that contains the vertex $(3,5)$ not in $\{ U_{\mathfrak A}\}$.
 Thus, there are  other chain factorizations corresponding to the same transposition $\sigma$ above, that use the chain $c'$ to the vertex $(3,3,1)$. 
\end{example}
\par We begin the proof of our main results with the special case of $2$-$U$-chains to illustrate our method. Recall from \eqref{Aeq} that $s_i$ and $t_i$ are the coefficients of $A_{\sf R}$ on
$\beta_i$ and $\alpha_i$, respectively. 
\par Given a chain factorization $f=(\nu_f,\sigma_f, C_f)$  of the monomial $\nu$ we may write
\begin{equation}\label{nufacteq}
f=f_1\cdot f_2\cdots f_s, \quad \nu_f=\nu_{1,f}\cdot \nu_{2,f}\cdots \nu_{s.f},
\end{equation}
where $f_\Upsilon$ collects all elements of $C_f$ to vertices $v$ in $U_{\mathfrak A,\Upsilon}$. We similarly write $g_{\mathfrak A}=g_{\mathfrak A,1}\cdots g_{\mathfrak A,s}$ and $\mu_{\mathfrak A}=\mu_{\mathfrak A ,1}\cdots \mu_{\mathfrak A ,s}$ or $\mu_{\mathfrak A}=\mu_1\cdots \mu_s$ for short. Here $\sigma_f $ determines for each vertex $v=v_{\mathfrak A,\Upsilon ,u}$ in the $\Upsilon$ component chain (a vertex corresponds to a column of $M_{\mathfrak A}$), the corresponding row $\eta^{-1}\sigma_f^{-1}v={A_{\sf R}}^{u'}( v_{\mathfrak A, \Upsilon ',1})$ for a suitable power $u'$ and index $\Upsilon '$. The sub-factorization $f_{\Upsilon '}$ determines a length $u'$ chain in $\mathrm{Diag}^{aug}(\mathcal D_P)$ from  $v_{\mathfrak A, \Upsilon ',1}$ to $v$ and a corresponding  monomial factor of $\nu_\Upsilon$.  
\begin{proposition}\label{2lem} Let $U_{\mathfrak A}=U_{a,b}$ be a maximal 2-$U$-chain in the augmented diagram of $\mathcal D_P$ and suppose that $A_{\sf R}\in \Mat_n(\sf R)\cap \mathcal U_{B,\sf R}$ is  simply adequate (Definition \ref{simplyadequatedef}).  Let $f=(\nu_f,\sigma_f, C_f)$ be a $\mathfrak A$-factorization of a monomial $\nu=\pm\mu_{\mathfrak A}$. Then $f=g_{\mathfrak A}$, that is, $\sigma_f = id$ and every $c_{\Upsilon ,u,f}$ is the standard chain from $v_{\mathfrak A,\Upsilon ,1}$ to $v_{\mathfrak A,\Upsilon ,u+1}$. The coefficient of $\mu_{\mathfrak A}$ in $\det(M_{\mathfrak A})$ is $1$. 
\end{proposition}
\begin{proof}
We will show that $g_\mathfrak A$ is the unique chain factorization $f=(\nu_f,\sigma_f,C_f)$ for a monomial $\nu$ of $\det(M_{\mathfrak A})$ such that $\nu$ has both the minimum possible multiplicity of $s_{a-1}$  (or of $s_a$ if $a$ is a singleton level of $S_P$), and the maximum possible multiplicity of $s_b$ (or of $z_b$ if $b$ is a singleton level of $S_P$). \par
Let the monomial $\nu=\pm\mu_{\mathfrak A}$ in the expansion of $\det M_{\mathfrak A}$ have 
chain factorization $f=(\nu_f,\sigma_f,C_f)$ as in Definition \ref{factornotation}. We write $\nu=\nu_1\cdot \nu_2 $ as in \eqref{nufacteq}.  We need to show that $f=g_{\mathfrak A}$.
We have $a\ge b+2$, so the almost rectangular levels of $\mathfrak A$ are
$(a,a-1,b,b-1)$ or, in the special case of a length-3 spread $(a,a-1,a-2)$, or also $(a,b,b-1), (a,a-1,b)$ when there is a singleton row. The component chains of $U_{\mathfrak A}$  are $U_b$ and $U_{\mathfrak A,2}$. \par
\vskip 0.25cm\noindent
{\bf Claim A}. $\sigma_f (U_b)\subset U_b$ and $s_b,z_b\nmid \nu_2$.
\vskip 0.25cm\noindent
{\emph{Proof of claim.}}
Assume first that $a$ is not a singleton level of $S_P$. The multiplicity of $s_{a-1}$ as a factor of $\mu_{\mathfrak A}$ and also of $|\nu|= \mu_{\mathfrak A}$ is
\begin{equation}\label{pows1eq}
  \left(bn_b+(b-1)n_{b-1}\right)+2\cdot\sum_{b<c\le a-2} n_c+\sum_{a-2<c} n_c,
\end{equation}
where the first summand comes from chains to vertices in the $b$ and $b-1$ levels of $U_b$, the middle from chains to vertices in the left hook and right hook that are below level $a-1$, and the right summand from chains to vertices in the top of the right hook of $U_b$. We will show that the sum  of all but the last term in \eqref{pows1eq} is a lower bound for the multiplicity of $s_{a-1}$ in any monomial of $\det (M_{\mathfrak A})$; and equality in \eqref{pows1eq} will greatly restrict the chain factorization $f$ of  $\nu=\pm \mu_{\mathfrak A}$. By \eqref{1eq} there is a unique saturated chain between two vertices $v=(0,i,a)$ and $v'=(0,i',a')$ at the extreme left of $\mathcal D_P$: there can be no return chain from a vertex $v''=(u'',i'',a'')$ with $u''>0$ to  $v'$. It follows that the chains encoded by $f$ to the left-hook vertices of $U_b$ are the same as those to the same vertices encoded by $g_{\mathfrak A}$. Since $A_{\sf R}$ is simply adequate the chain encoded by $f$ to each vertex  of the $b,b-1$ levels of $U_b$ must have at least one $\beta_{a-1}$ edge encoded by $s_{a-1}$ (see Note \ref{saturatednote}). \par Similarly for each vertex of the right hook of $U_b$, lying at level $a-2$ or below.  There are $\sum_{c\ge a-1}n_c$ vertices at the top  of the right hook of $U_b$ that might be reached by a chain encoded by $f_1$ lying entirely on or above the $a-1 $ levels. Suppose now that $\kappa\le \sum_{c\ge a-1}n_c$ chains encoded by $f_2$ to vertices of $U_{\mathfrak A,2}$ dip to the $a-2$ or lower level: each such chain contributes an $s_{a-1}$ factor for $\nu_2$. Then $\kappa$ chains of $f_1$ to the top vertices of the right hook of $U_b$ must lie entirely at or above the $a-1$ level, in order for the power of $s_{a-1}$ dividing $\nu$ to not exceed the value given by \eqref{pows1eq} for $\mu_{\mathfrak A}$.  \par
We now compare the factors $s_b^k$ encoded by $f$ and by $g_{\mathfrak A}$.
The multiplicity of $s_b$ as a factor of $\mu_{\mathfrak A}$ is the same as in $\mu_1$ and is
\begin{equation}\label{sbeq}
\left( b(b-1)/2\right)\cdot\left( n_b+n_{b-1}\right)+(b-1)\cdot\sum_{c>b}n_c.
\end{equation}
Here the first summand is from chains to vertices in the $b,b-1$ levels, and the second counts $b-1$ occurences of $s_b$ in each chain to a vertex of the right hook of $U_b$. We will see that the sum in \eqref{sbeq} is an upper bound for the multiplicity of $s_b$ as a factor of any monomial of $\det (M_{\mathfrak A})$ containing exactly the power of $s_{a-1}$ specified in \eqref{pows1eq}, so equality in  $\nu=\pm \mu_{\mathfrak A}$ will further restrict $f$. For each vertex of $U_b$, the power of $s_b$ encoded by a chain of $f_1$ to that vertex is no greater than the power of $s_b$ encoded by the standard chain of $g_{\mathfrak A,1}$ to the same vertex.
 The highest power of $s_b$ that a chain $p$ encoded by $f_2$ could contribute to $\nu_2$ is $s_b^{b-2}$, since $U_{\mathfrak A,2}$ lies to the left of the rightmost column $0:B$ of $\mathcal D_P$: by \eqref{2eq} there can be no chain from $(b-1,b-1,k)$ to
$(u,c,k')$ for $u<c$ when $b<c$.\footnote{Corollary \ref{bpowcor} and \eqref{bpoweq} give a sharper bound on the multiplicity of the edge $\beta_b$.} Thus, when one replaces $\kappa$ standard chains of $g_{\mathfrak A,1}$  to top vertices of the right hook of $U_b$ by $\kappa$ chains whose levels are entirely on or above the $a-1$ level in $f_2$, and makes up the missing $s_{a-1}^\kappa$ power in $\nu=\pm \mu$ by adding $\kappa$ chains of $f_2$ dipping to the $b-1$ level, one loses at least a total multiplicity of $\kappa$ for $s_b$ in $\nu$ in comparison to the multiplicity of $s_b$ in $\mu$ given by \eqref{sbeq}.  It follows that $\kappa=0$.
 \par Thus $s_b\nmid f_2$. If any component chain for $f_1$ began from $(2,a,1)=v_{\mathfrak A,2,1}$, it would contribute at least one less power of $s_b$ to $\nu_1$ than the the power contributed to $\mu_1$ by the standard chain from $(1,a,1)$ to the corresponding vertex encoded by $g_{\mathfrak A,1}$, again by \eqref{1eq}. Furthermore, there is no way to increase the total $s_b$ power by choosing different chains than the standard chains encoded by $g_1$ to the vertices of $U_b$. This implies that all chains encoded by $f_1$ begin from $(1,a,1)$: equivalently, $\sigma_f (U_b)\subset U_b$. \par
In the special case that the $a$ level is a singleton we replace $s_{a-1}$ by $s_a$ above. Since $\mathfrak A$ is maximal, the case of $b$ being a singleton level occurs only when $b$ is the minimum level of an odd-length MCS of $S_P$ (Lemma \ref{singletonlem}.) That we have included the added variables $z_\ell$ in \eqref{Aeq} allows us to carry out the second part of the argument,  replacing $s_b$ by
$z_b$.  and using \eqref{sbeq} as lower bound for the multiplicity of $z_b$ in $\nu$, when $b$ is a singleton level in $\mathfrak A$. This completes the proof of Claim A.
\vskip 0.25cm\noindent
{\bf Claim B.} The restriction of $\sigma_f $ to $ U_b$ is the identity, and $f_1=g_{\mathfrak A,1}$.\par\noindent
 The chains encoded by  $f_1 $ and $g_{\mathfrak A,1}$ agree for vertices of the left or right hook of $U_b$. On the left hook by uniqueness of the chains to vertices of the left hook; and on the right hook by the argument above requiring each such vertex $v$ to contribute $s_b^{b-1}$ to $\mu_1$ -- the only way to do so is for the chain to the vertex $v$ to pass through $(b-1,b-1,1)$: then it is the standard chain to $v$. Since $s_b\nmid \nu_2$ by Claim A,
in order for the power of $s_b$ given by \eqref{sbeq} to divide $\nu$, the
chain encoded by $f_1$ to each vertex $v=(u,b,k)$ must contribute $u-1$ and that to $v=(u,b-1,k)$ must contribute $u$ to the $s_b$ power of $\nu_1$: the only way to do so is for each such chain to pass through $(1,b,1)$, and to be the standard chain to $v$. This completes the proof of Claim~B.\vskip 0.25cm\noindent
{\bf Claim C}. The restriction of ${\sigma_f}$ to $ U_{\mathfrak A,2}$ is the identity, and $f_2=g_{\mathfrak A,2}$.\par\noindent
We have shown that $\sigma_f (U_{\mathfrak A,2})\subset  (U_{\mathfrak A,2})$. Since $\nu=\mu$ and $\nu_1=\mu_1$ we have $\nu_2=\mu_2$;  since the factorization of $\nu_1 $ is that of $\mu_1$, all chains contributing factors to  $\nu_2$ must start from $(2,a,1)$ and lie entirely within the chain $U_{\mathfrak A,2}$. It follows similarly to the proof of Claim B that $\sigma_f$ on $U_{\mathfrak A,2}$ is the identity, and $f_2=g_{\mathfrak A,2}$.\par
This completes the proof of the Proposition.
\end{proof}
\begin{example}\label{propex} Let $P$ have $n_i>0$ parts $i$ for $1\le i\le 5$ and consider the 2-$U$-chain $U_{\mathfrak A}=U_{4,2}$ of 
$\mathcal D_P$.  By \eqref{pows1eq} the power of $s_3$ dividing the $\mu_{\mathfrak A}$ term of $\det M_{\mathfrak A}$ is $(2n_2+n_1)+(n_5+n_4+n_3)$.  By \eqref{sbeq}  the power of $s_2$ dividing the $\mu_{\mathfrak A}$ term of $\det M_{\mathfrak A}$ is  $(n_2+n_1)+(n_5+n_4+n_3)$: this is the maximum power of $s_2$ possible for terms containing exactly the power of $s_3$ given by \eqref{pows1eq} and only $u_{\mathfrak A}$ attains this maximum. For $P=(5,4,3,3,2,1)$ of Example \ref{3.4ex} and Figure \ref{mufigure} these powers are $s_3^7$ and $s_2^6$.
\end{example}
\begin{theorem}\label{sthm}  Let $U_{\mathfrak A}$ be a maximal s-$U$-chain in the augmented diagram of $\mathcal D_P$ and suppose that $A_{\sf R}\in \Mat_n({\sf R})\cap\, \mathcal U_{B,\sf R}$ is  simply adequate.   Let $f=(\nu_f,\sigma_f, C_f)$ be a $\mathfrak A$-factorization of $\nu=\pm\mu_{\mathfrak A}$. Then $f=g_{\mathfrak A}$. The coefficient of $\mu_{\mathfrak A}$ in $\det(M_{\mathfrak A})$ is $1$. 
\end{theorem}
\begin{proof} We will show this by induction on $s$. The case $s=1$ is essentially Claim~B of Proposition \ref{2lem}, and the case $s=2$ is Proposition \ref{2lem}. Suppose that the theorem is known for $s-1$ and all partitions $P$. Let the monomial $\nu=\pm\mu_{\mathfrak A}$ in the expansion of $\det M_{\mathfrak A}$ have 
chain factorization $f=(\nu_f,\sigma_f ,C_f)$ as in Definition \ref{factornotation}. Let $a=a_1, b=a_s$. We will show first that $\sigma_f (U_{b})\subset U_{b}$, then that the restriction $\sigma_f$ to $ U_{b}$ is the identity, and $f_1=g_{\mathfrak A,1}$. Then induction suffices to complete the result.\vskip  0.20 cm\noindent
{\bf Claim A.} $\sigma_f (U_b)\subset U_{b}$ and $s_b,z_b\nmid \nu_2$, \par\noindent
Assume first that $m=a_{s-1}$ is not a singleton of $S_P$, and consider the variable $s_{m-1}$. The multiplicity of $s_{m-1}$ in $\mu_{\mathfrak A}$ is
\begin{equation}\label{pows2eq}
 \left(bn_b+(b-1)n_{b-1}\right) +2\cdot\sum_{b<c\le m-2} n_c+\sum_{m-2<c} n_c.
\end{equation}
We will see that equality in \eqref{pows2eq} for $\nu_f=\pm\mu_{\mathfrak A}$ will greatly restrict $f$.
The chains encoded by $f$ to the left hook vertices of $U_b$ are the same as those encoded by $g_{\mathfrak A}$ to the same vertices. By the definition of $A_{\sf R}$ the chain encoded by $f_1$ to each vertex of $U_b$ at the $b,b-1$ levels must have at least one $\beta_m$ edge encoded by an $s_m$ factor. Likewise for the chain encoded by $f_1$ to any right hook vertex at level $m-2$ or below. However, $\kappa\le \sum_{c\ge m} n_c$ chains encoded by $f$ might lie entirely at level $m-1$ or above; the missing $s_{m-1}$ powers for these chains must be replaced by those in $\kappa$ chains encoded by $f_2,\ldots ,f_s$ to vertices in $U_{\mathfrak A}-U_b$. 
\par The multiplicity of $s_b$ as a factor of $\mu_{\mathfrak A}$ is given by \eqref{sbeq}. The right hook vertices of $U_b$ each contribute 
$ s_b^{b-1}$ to $\mu_1$. Any chain to a vertex of one of the top 
$s-1$ component chains $\mu_{\mathfrak A,i}, 2\le i\le s$ can contribute at most $s_b^{b-2}$ to $\nu$, by \eqref{2eq} or \eqref{bpoweq}.  As before for $s=2$, we conclude that $\kappa=0$, that $s_b\nmid \nu_2\cdots \nu_s$, and that $\sigma_f\mid U_b\subset U_b$. In the special case that the $a$-level is a singleton we replace $s_{a-1}$ by $s_a$ above; in case $U_b$ is a singleton level, we replace $s_b$ by $z_b$ in the above argument, as in the proof of Claim A of Proposition~\ref{2lem}.
This proves Claim A.\par
The same argument as in the proof of Proposition \ref{2lem} Claim B now shows that $\sigma_f{\mid U_b}$ is the identity, and $f_1=g_{\mathfrak A,1}$. We have also shown that $\sigma_f (U_{\mathfrak A}-U_b)\subset  U_{\mathfrak A}-U_b$, and that the portion of the factorization of $\nu$ coming from vertices in $\{U_{\mathfrak A}-U_b\}$ involves no edges of $\mathcal D_P$ below level $m-1$. Since these chains encoded by $f$ to vertices of $U_{\mathfrak A}-U_b$ start and end in $U_{\mathfrak A}-U_b$ they don't involve edges in the hooks of $U_b$.  \par
Now peel the chain  $U_b=U_{\mathfrak A,1}$ from $U_{\mathfrak A}$ to form $U_{\mathfrak A'}, \mathfrak A'=(a_1,\ldots ,a_{s-1})$, and regard its image   $s-1$ chain $U'$ with label ${\mathfrak A}'_{P'}=(a_1-2,\ldots ,a_{s-1}-2)$ in $\mathcal D_{P'}$, where $P'$ is obtained from $P$ by  peeling off $U_b$ and ommitting parts below $U_b$:
\begin{equation*}
 n_i(P')=\begin{cases}&n_{i+2}(P) \text { for $i\ge b-1$}\\
&0 \text { for $i\le b-2$}.
\end{cases} 
\end{equation*}
Since $U_b=U_{\mathfrak A,1}$ was an outside chain, $\mathcal D_{P'}\subset \mathcal D_P$.
The induction step applied to $U'$ and $\mathcal D_{P'}$ now shows that the portion $f_2\cdot f_3\cdots f_s$ of the factorization $f$ corresponding to vertices of $U_{\mathfrak A}-U_b$ agrees with the factorization $g_2\cdot g_3\cdots g_s$ of the corresponding portion of $g=g_{\mathfrak A}$. Putting this together with  $f_1=g_{\mathfrak A,1}$ we conclude that 
$f=g_{\mathfrak A}$. This completes the proof of the  induction step and the Theorem.
\end{proof}
\begin{corollary}\label{maincor} Let $\sf k$ be an infinite field, let ${\sf R}$ be the polynomial ring of \eqref{Aeq}, and suppose that $A_{\sf R}\in \End_{\sf R}V_{\sf F}$ is the simply adequate element of $\,\mathcal U_{B,\sf R}$. Then the Jordan partition $P_{A_{\sf R}}$ over the quotient field $\sf F$ of $\sf R$ satisfies,  $P_{A_{\sf R}}\ge \lambda_U(\mathcal D_P)$.
\end{corollary}
\begin{proof}
 Let $A_{\sf R}\in\mathcal U_{B,\sf F}$ be simply adequate.
Let $U_{\mathfrak A (s)}$ be a maximum-length $s$ $U$-chain of $\mathcal D_P, 1\le s\le r_P$. Theorem \ref{sthm} shows that for each $s$, the projections $\pi_{\mathfrak A (s)}:\mathcal T_{\mathfrak A (s), A}\to \mathcal U_{\mathfrak A (s)}$  have the maximum possible rank $\mid U_{\mathfrak A (s)}\mid$. By Lemma~\ref{ranklem} the partition $P_{A_{\sf R}}$ has first $s$ parts summing to at least $\mid U_{\mathfrak A (s)}\mid $. By Definition \ref{lambdaUdef} and \eqref{orbitclosureeq} this implies that over the quotient field $\sf F$ we have $P_{A_{\sf R}}\ge  \lambda_U(\mathcal D_P)$.
\end{proof}
\begin{theorem}\label{mainthm} Let $\sf k$ be an infinite field. Then
$Q(P)\ge \lambda_U(\mathcal D_P)$. In particular, there is an adequate $A\in \mathcal U_B$ over $\sf k$ satisfying $P_A\ge \lambda_U(\mathcal D_P)$.
\end{theorem}
\begin{proof}  Let $A_{\sf R}\in\mathcal U_{B,\sf R}$ be simply adequate.
Let $U_{\mathfrak A (s)}$ be a maximum-length $s$ $U$-chain of $\mathcal D_P, 1\le s\le r_P$. Theorem \ref{sthm} shows that when $\det (M_{\mathfrak A (s)})$ is expanded into a sum of monomials of $\sf R$ over $\sf k$ corresponding each to a chain factorization $f=(\nu_f, \sigma_f, C_f), C_f=\{c_{iu}\})$ as in Definition \ref{factornotation},  there is a unique term $\mu_{\mathfrak A (s)}$. Since $\sf k$ is an infinite field, we can choose $\theta: \sf F \to \sf k$, that is, substitute for the variables of $\sf R$, so that each $\theta (\det(M_{\mathfrak A (s)}))\not=0$. Thus, $A=\theta (A_{\sf R})\in \Mat_n({\sf k})\cap \mathcal U_B$ satisfies
rank$(\pi_{\mathfrak A (s)})= | U_{\mathfrak A (s)}|$ on $\mathcal T_{\mathfrak A(s),A}$  for each $U_{\mathfrak A (s)}$. As in Corollary~\ref{maincor}, this implies $P_{A}\ge \lambda_U(\mathcal D_P)$. By the irreducibility of $\mathcal U_B$, we have $Q(P)\ge \lambda_U(\mathcal D_P)$.
\end{proof}\par
 Write $Q(P)_{\sf k}=P_A$ for a generic $A\in \mathcal N_B, B=J_P$ over the field $\sf k$. Write  $Q(P)_{\mathbb R}$ over the reals $\mathbb R$ as  $Q(P)_{\mathbb R}=(q_1(\mathbb R),\ldots ,q_{r(P)}(\mathbb R))$ where $ q_1(\mathbb R)\ge q_2(\mathbb R)\ge\ldots $, and write $\lambda_U(\mathcal D_P)=(\lambda_{1,U}(P),\lambda_{2,U}(P),\ldots)$ where $ \lambda_{1,U}(P)\ge \lambda_{2,U}(P)\ge,\ldots$. 
\begin{corollary}\label{charcor} Let $\sf k$ be in infinite field. Fix a partition $P\vdash n$ an an integer $k, 1\le k\le r_P$.  Assume that $\sum_{1=1}^k q_i(\mathbb R)=\sum_{1=1}^k \lambda_{i,U}(P)$. Then the analogous equality holds for $Q(P)_{\sf k}$.
\end{corollary}
\begin{proof} This follows from $Q(P)_{\mathbb R}\ge Q(P)_{\sf k}$ and Theorem \ref{mainthm}.
\end{proof}\par
The second author has shown
\begin{uthm}\cite{Kha2}\label{Kcor} The minimum part of $\lambda(\mathcal D_P)$ is equal to the minimum part of $\lambda_U(\mathcal D_P)$.
\end{uthm}
This together with Theorem \ref{mainthm} and \eqref{qlambdaeq} show
\begin{corollary}\cite{Kha2} Let $\sf k$ be an infinite field. The minimum part of $Q(P)$ is equal to the minimum part $m_P$ of $\lambda_U(\mathcal D_P)$.
\end{corollary}
An explicit formula for $m_P$ in terms of $P$ is given in \cite{Kha2}. Our result also has the corollary of extending  
P. Oblak's Theorem \ref{obthm} to an infinite field $\sf k$. These show
 
\begin{corollary}(\cite{Obl1}\label{main2cor} $r_P=2$,\cite{Kha2} $r_p= 3$).\label{3.9cor} Let $\sf k$ be an infinite field. When $r_P\le 3, Q(P)=\lambda_U(\mathcal D_P)=\lambda (\mathcal D_P)$ and can be explicitly written in terms of $P$.
\end{corollary}
\begin{example}\label{rp=3ex}  For $P= (5,4,3^3,2^3,1^2), r_P=3$. The maximum-length simple $U$-chains are $U_3$ of length $|U_3|=3(3)+3(2)+2(2)=19$, where the two hooks each have length two,  and also $U_2$. The maximum-length 2-$U$ chain is $U_{4,2}$ of length $|U_{4,2}|=25$. So $ Q(P)= (19,6,1)$.
\end{example}
\begin{figure}
\begin{equation*}
M_{\mathfrak A}=\left(
\begin{array}{c|ccccccc}
& \sf a&\sf e&\sf g&\sf f&\sf d&\sf b&\sf c\\
\hline
1\otimes_{\sf F}  \sf a&1&0&\ldots &&&&0\\
x\otimes_{\sf F}  \sf a&0&s_4&0& 0&0&z_4&0\\
x^2\otimes_{\sf F}  \sf a&0&0&s_2s_4&s_4z_4&0&0&t_4s_4+z_4^2\\
x^3\otimes_{\sf F}  \sf a&0&0&0&t_2s_2s_4&t_4s_4z_4+z_4t_4s_4+z_4^3&0&0\\
x^4\otimes_{\sf F}  \sf a& 0&0&0&0&t_4t_2s_2s_4&0&0\\
1\otimes_{\sf F}  \sf b&0&&\ldots &&0&1&0\\
x\otimes_{\sf F}  \sf b&0&0&0&s_4&0&0&z_4\\
\end{array}
\right)
\end{equation*}
\vspace{-0.5cm}
\protect\caption{Matrix $M_{\mathfrak A}$ for $\mathfrak A=(4,2), P=(4,2,1)$.\label{421fig}}
\end{figure} 
\begin{example}\label{421advex} Recall from Example \ref{421ex} that for $P=(4,2,1)$  we have $r_P=2$;  from this and Oblak's index formula \eqref{obeq}   we have $Q(P)=(5,2)$. We use the notation of Example~\ref{421ex} and Figure \ref{421afig} for the basis $\sf 
B$ of $V$. The simply adequate $A_{\sf R}$ of \eqref{Aeq} and  Corollary~\ref{maincor} with coefficients in $\sf R$ satisfies
\begin{equation}\label{421adveq}
A_{\sf R}\cdot {\sf a}=z_4{\sf b}+s_4{\sf e}, \,A_{\sf R}\cdot {\sf e}=t_4{\sf c}+s_2{\sf g}, \,A_{\sf R}\cdot {\sf g}=t_2{\sf f},
\end{equation}
where $s_4,s_2,t_4,t_2,z_4$ are the variables of $\sf R$. Since $A_{\sf R}$ commutes with $B$, these determine $A_{\sf R}$. The matrix $M_{\mathfrak A}$ is given in Figure \ref{421fig}; the entries can be obtained from Figure~\ref{421DPfig}
by multiplying the variables of $\sf R$ labelling the edges of the chain corresponding to each entry.
\begin{figure}[htb]
\begin{center}
\includegraphics[scale=.53]{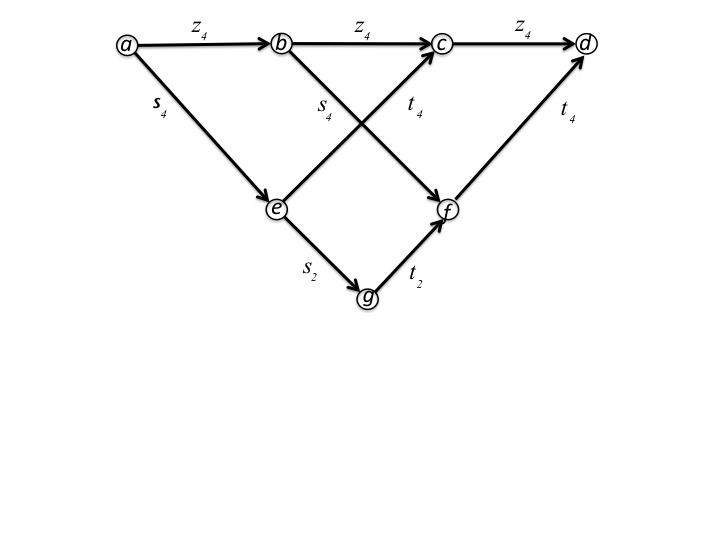}
\end{center}\vskip -4.6cm
\protect\caption{ $\mathrm{Diag}(\mathcal D_P)$ and variables in $\sf R$ for  $P=(4,2,1)$.} \label{421DPfig}
\end{figure}\vskip 0.4cm The determinant of $ M_{\mathfrak A}$ has a unique non-zero term: its unique chain factorization arises from $\mu_{\mathfrak A}$:
\begin{equation*}
\det M_{\mathfrak A}=\mu_{\mathfrak A}=1\cdot s_4\cdot s_2s_4\cdot t_2s_2s_4\cdot t_4t_2s_4s_2\cdot 1\cdot z_4=s_4^4s_2^3t_2^2t_4z_4.
\end{equation*}
When, as here, there is a unique maximum length chain from the source $\sf a$ to the
sink $\sf d$ of $\mathcal D_P$, any matrix $A$ as in \eqref{Aeq} such that the values of $s_i, t_i,t_{i,k}, z_\ell$ are non-zero in $\sf k$ satisfies, the maximum part of the Jordan type of $A$ is the index ${\mathrm i}(Q(P))$.
Here, setting the variables of $\sf R$ of \eqref{421adveq} equal to $1$ yields the matrix $A\in \mathcal U_B$ of \eqref{421matrix1eq} satisfying $\dim \langle \sf k[A]\cdot\{ \sf a ,\sf b\}\rangle=7$ and $\dim \langle \sf k[A]\cdot\{ \sf a \}\rangle=5$. \par Also the matrix $\Mat_{\mathfrak A'}$ for $\mathfrak A'=(2)$ is the leading $5\times 5 $ minor of $\Mat_{\mathfrak A}$, with determinant the monomial $\mu_{\mathfrak A,1}$ in $\sf R$.
This shows that here $P_A=Q(P)=(5,2)$, as stated in Example~\ref{421ex} and Corollary~\ref{main2cor}. \par

\end{example}

\begin{remark} 

Even if the questions of Section \ref{intsec} be answered, it still appears subtle to understand,  given a stable partition $Q$, the set of partitions $P$ such that $Q(P)=Q$: see \cite{Obl2} for some results  and open problems in this direction.\par
We have wondered why this problem of understanding the map $P\to Q(P)$ was not posed much earlier in the literature.  Perhaps it was supplanted by another natural problem, to characterize maximal vector spaces of commuting matrices \cite{Sup}. \par
Recent work of E. Friedlander, J. Pevtsova, and A. Suslin on modular representations has involved both Jordan types and the variety of commuting nilpotent matrices  \cite{FPS}.
\end{remark}
\normalsize

\begin{remark}[The field $\sf k$]\label{fieldrem}
J. R Britnell and M.~Wildon have shown that over the finite field ${\sf k}(p^r)$ having $p^r$ elements, the Jordan types $P_A=(d+1,d-1)$ and $ P_B=(d,d)$ occur for two commuting nilpotent matrices $A,B$  if and only if $d$ is not divisible by $p(p^{2r}-1)/2$ for $p>2$, and by $2(4^r-1)$ when $p=2$  \cite[Proposition~4.12]{BrWi}. However, when $\sf k$ is infinite, they show that there are always commuting matrices $A,B$ in these two orbits \cite[Remark 4.15]{BrWi}.\par G.~McNinch showed in \cite[Example 22]{McN}
(see also \cite[Example 2.18]{BI}) that the class of a generic linear combination $A+tB$ of certain nilpotent $A,B$ in a tensor product
$V=V_d\otimes V_2$ over an infinite field $\sf k$ depends on the characteristic of $\sf k$: this class is  $(d+1,d-1)$ for $d$ invertible, but is $(d,d)$ when $d$ divides $ \cha \sf k$.
 It appears to be open whether
the set of pairs of Jordan partitions for the similarity classes of two commuting nilpotent matrices depends on $\cha \sf k$ when $\sf k$ is an infinite field.\par

\end{remark}

\vskip 0.4cm\noindent
e-mail:  
a.iarrobino@neu.edu\\
khatamil@union.edu
\end{document}